\documentclass[10pt]{amsart}
\usepackage{amsmath}
\usepackage{amssymb}
\usepackage{amscd}
\usepackage{tikz-cd}
\usepackage[all]{xy}
\usepackage{xcolor}
\newcounter{TmpEnumi}
\numberwithin{equation}{section}

\def\today{\number\day\space\ifcase\month\or   January\or February\or
   March\or April\or May\or June\or   July\or August\or September\or
   October\or November\or December\fi\   \number\year}

\theoremstyle{definition}
\newtheorem{thm}{Theorem}[section]
\newtheorem{lem}[thm]{Lemma}
\newtheorem{prp}[thm]{Proposition}
\newtheorem{dfn}[thm]{Definition}
\newtheorem{cor}[thm]{Corollary}

\newtheorem{rmk}[thm]{Remark}
\newtheorem{ntn}[thm]{Notation}
\newtheorem{exa}[thm]{Example}

\newcommand{\beq}{\begin{equation}}
\newcommand{\eeq}{\end{equation}}
\newcommand{\beqr}{\begin{eqnarray*}}
\newcommand{\eeqr}{\end{eqnarray*}}
\newcommand{\bal}{\begin{align*}}
\newcommand{\eal}{\end{align*}}
\newcommand{\bei}{\begin{itemize}}
\newcommand{\eei}{\end{itemize}}

\newcommand{\af}{\alpha}

\newcommand{\gm}{\gamma}
\newcommand{\dt}{\delta}
\newcommand{\ep}{\varepsilon}
\newcommand{\zt}{\zeta}
\newcommand{\et}{\eta}

\newcommand{\io}{\iota}
\newcommand{\te}{\theta}
\newcommand{\ld}{\lambda}
\newcommand{\sm}{\sigma}
\newcommand{\kp}{\kappa}
\newcommand{\ph}{\varphi}
\newcommand{\ps}{\psi}
\newcommand{\rh}{\rho}
\newcommand{\om}{\omega}

\newcommand{\Q}{{\mathbb{Q}}}
\newcommand{\Z}{{\mathbb{Z}}}
\newcommand{\R}{{\mathbb{R}}}
\newcommand{\C}{{\mathbb{C}}}
\newcommand{\N}{{\mathbb{Z}}_{> 0}}
\newcommand{\Nz}{{\mathbb{Z}}_{\geq 0}}

\newcommand{\id}{{\operatorname{id}}}

\newcommand{\sint}{{\operatorname{int}}}

\newcommand{\spec}{{\operatorname{sp}}}

\newcommand{\supp}{{\operatorname{supp}}}

\newcommand{\Aut}{{\operatorname{Aut}}}
\newcommand{\Ad}{{\operatorname{Ad}}}
\newcommand{\Aff}{{\operatorname{Aff}}}

\newcommand{\mdim}{{\operatorname{mdim}}}

\newcommand{\Cu}{{\operatorname{Cu}}}
\newcommand{\W}{{\operatorname{W}}}

\newcommand{\cZ}{{\mathcal{Z}}}

\newcommand{\Mi}{M_{\infty}}

\newcommand{\andeqn}{\qquad {\mbox{and}} \qquad}

\newcommand{\ds}[1]{{\displaystyle{#1}}}

\newcommand{\Wolog}{Without loss of generality}

\newcommand{\ifo}{if and only if}

\newcommand{\ca}{C*-algebra}
\newcommand{\uca}{unital C*-algebra}

\newcommand{\hm}{homomorphism}

\newcommand{\tst}{tracial state}
\newcommand{\hsa}{hereditary subalgebra}

\newcommand{\pj}{projection}

\newcommand{\nzp}{nonzero projection}

\newcommand{\ct}{continuous}
\newcommand{\cfn}{continuous function}

\newcommand{\cms}{compact metric space}

\newcommand{\chs}{compact Hausdorff space}
\newcommand{\hme}{homeomorphism}
\newcommand{\mh}{minimal homeomorphism}

\renewcommand{\S}{\subseteq}
\newcommand{\ov}{\overline}
\newcommand{\SM}{\setminus}
\newcommand{\I}{\infty}
\newcommand{\E}{\varnothing}
\newcommand{\Kt}{K \otimes}

\newcommand{\set}[1]{\left\{ #1 \right\}}
\newcommand{\cc}[1]{\langle #1 \rangle}
\newcommand{\inn}[1]{\left\langle #1 \right\rangle}
\newcommand{\ccA}[1]{\langle #1 \rangle_A}
\newcommand{\ccB}[1]{\langle #1 \rangle_B}
\newcommand{\ten}{\otimes}
\newcommand{\cuntz}{\precsim}
\newcommand{\JS}{\mathcal{Z}}
\newcommand{\les}{<_{\operatorname{s}}}
\newcommand{\ccAf}[1]{\langle #1 \rangle_{A^{\alpha}}}
\newcommand{\RD}{{\operatorname{dim}}_{\operatorname{Rok}}^{\operatorname{c}}}

\title[Pureness of Crossed Products]{Pureness of Certain Crossed Product C*-Algebras}

\author[D.\ Archey, J.\ Buck, J.\ Mohammadkarimi, N.C.\ Phillips
and A.\ Seth]{Dawn Archey, Julian Buck, Javad Mohammadkarimi,
N.~Christopher Phillips, and Apurva Seth}

\address{Dawn~Archey, University of Detroit Mercy, Department of Mathematics, 4001 W. McNichols Rd., Detroit MI 48221, USA.}
\email{archeyde@udmercy.edu}

\address{Julian~Buck, Department of Mathematics, Okanagan College, 1000 KLO Road, Kelowna BC, V1Y 4X8 Canada.}
\email{jbuck@okanagan.bc.ca}

\address{Javad~Mohammadkarimi, Research Center for Operator Algebras, East China Normal University,  Shanghai
200062, China.}   
\email{javad.mohammadkarimi1369@gmail.com}

\address{N.~Christopher Phillips, Department of Mathematics, University  of Oregon,
	Eugene OR 97403-1222, USA.}

\address{Apurva~Seth, Mathematical Institute, University of Oxford, Radcliffe Observatory, Andrew Wiles Building, Woodstock Rd, Oxford OX2 6GG, United Kingdom.}
\email{apurvaseth14@gmail.com}  


\date{10~April 2026}

\thanks{ The third author was partially supported by the Research Center for Operator Algebras and by the Science and Technology Commission of Shanghai Municipality
(No. 22DZ2229014) at East China Normal University and the fifth author was supported by the Engineering and Physical Sciences Research Council (EP/X026647/1). For the purpose of Open Access, the authors have applied a CC BY public copyright license to any Author Accepted Manuscript (AAM) version arising from this submission.}

\begin{document}

\maketitle

\begin{abstract}
We establish comparison and divisibility properties for
crossed product C*-algebras arising from automorphisms of algebras
$C (X, D)$ which lie over minimal homeomorphisms,
from actions of compact groups which have
finite Rokhlin dimension with commuting towers,
and from actions of compact groups which have the
restricted tracial Rokhlin property with comparison.
We deduce that these crossed products we
consider are pure, and conclude they have stable rank one,
and in certain cases have real rank zero.
We give examples in which these properties
do not follow from previous results,
in the case of $C (X, D)$
due to the lack of $\JS$-stability of~$D$, the
underlying topological spaces not being finite dimensional, or both.
\end{abstract}

\section{Introduction}

Significant progress has been made on the classification of
crossed product \ca{s} arising from minimal dynamical systems.
The unpublished long preprint \cite{QLinPhDiff} of Q.~Lin
and N.~C.\  Phillips (see also the survey articles
\cite{QLinPh1} and \cite{QLinPh2}) provides a
description of the transformation group
\ca{s} arising from minimal diffeomorphisms of
finite dimensional smooth compact manifolds in
terms of a direct limit decomposition.
In \cite{HLinPh}
and \cite{TomsWinter}, it is shown that crossed
products arising from minimal homeomorphisms
of infinite compact metrizable spaces with finite covering
dimension are classified by their ordered K-theory in the
presence of sufficiently many projections (for instance,
when projections separate traces).
In \cite{TomsWinter}
it is further proved that crossed products by such minimal
homeomorphisms have finite nuclear dimension, and
hence absorb the Jiang-Su algebra $\cZ$ tensorially
(that is, are $\cZ$-stable).
Finally, G.~A.\  Elliott and
Z.~Niu (\cite{EllNiu}) have shown that crossed products
by minimal homeomorphisms of compact metric spaces
with mean dimension zero (including all minimal
homeomorphisms of finite dimensional compact
metric spaces) are ${\mathcal{Z}}$-stable, from which
it follows that they are classifiable in the sense of the
Elliott program by Corollary D of Theorem A of \cite{CETWW}.

There is also substantial work on crossed products
$C^* (\Z, A, \af)$ when $A$ is simple.

Much less is known when the \ca~$A$ is
neither commutative nor simple.
The easiest more general
case to consider is presumably $A = C (X, D)$, in which $X$
is a \cms{} and $D$ is a simple unital \ca.
In this case, every
automorphism of~$A$ ``lies over'' a \hme~$h$ of~$X$.
(See \cite{ArBcPh2}.)
We want the crossed product to be simple,
and this happens \ifo{} $h$ is minimal.
Hua (\cite{Hua}) has
shown that such crossed products have tracial rank zero
when moreover $X$ is the Cantor set, $D$ has tracial rank
zero, and some additional K-theoretic assumptions are made.
This case was more systematically studied in \cite{ArBcPh2},
in which several results were given showing that if $D$ has
good properties (for example, has stable rank one, has strict
comparison of positive elements, or is tracially $\cZ$-stable),
and other conditions are satisfied, then $C^* (\Z, C (X, D), h)$
has good properties.
In particular, by \cite[Theorem 4.5]{ArBcPh2},
if $D$ is tracially $\cZ$-stable, then so is $C^* (\Z, C (X, D), h)$,
even if $\mdim (h) > 0$.
If $D$ is nuclear, it then follows from
\cite[Theorem~4.1]{HirOr} that $C^* (\Z, C (X, D), h)$ is $\cZ$-stable.
Therefore $C^* (\Z, C (X, D), h)$ is classifiable
in the sense of the Elliott program.
However, there are many
interesting examples of crossed products of this sort in which
$D$ is not nuclear and not even tracially $\cZ$-stable.
Many examples are given in \cite[Section~6]{ArBcPh2}.
Many of the results there are limited to the case
$\dim (X) \leq 1$ or even $\dim (X) = 0$, and are
therefore somewhat artificial.
The reason is lack of
sufficient information about algebras of the form $C (Y, D)$,
or recursive subhomogenous algebras over~$D$, and their
direct limits, when the dimensions of the spaces are greater
than~$1$.

This is where pure C*-algebras come into play.
Following Winter \cite[Definition 3.6]{Wint2012}, a C*-algebra is said to be pure if it has strict comparison and is
almost divisible.
More generally, weaker notions of comparison and divisibility, namely $m$-comparison and $n$-almost divisibility, suffice: if $A$ has $m$-comparison and is $n$-almost divisible for some $m,n \in \mathbb{N}$, then $A$ is pure by~\cite[Theorem~5.7]{AnPrTlVt}.
See Section~\ref{S_Prelims} for the relevant definitions and further details.
We warn that $m$-comparison as used here is a much weaker condition
than the property
with the same name used to define the radius of comparison,
as for example in \cite[Definition 6.1]{Toms06}.
However, $0$-comparison is the same for both definitions.

Pureness is closely connected to $\cZ$-stability.
By \cite{Rord2004}, $\cZ$-stability implies pureness, and for separable C*-algebras, the two notions coincide at the level of central sequence algebras.
(See \cite{PerThiVil2025}.)
In the nonnuclear setting, $\cZ$-stability is often unavailable, whereas pureness remains accessible and continues to play a useful role in establishing strict comparison.
It is therefore natural to view pureness as one plausible extension of $\cZ$-stability to the case of nonnuclear simple \ca{s}.
Two other possible extensions are $\cZ$-stability itself and (in the unital case)
tracial $\cZ$-stability as in \cite[Definition~2.1]{HirOr}.
Certainly $\cZ$-stability implies tracial $\cZ$-stability.
It seems likely that tracial $\cZ$-stability implies
pureness, although, as far as we know, nobody has
actually proved this.
For simple separable nuclear
unital \ca{s}, tracial $\cZ$-stability implies
$\cZ$-stability (\cite[Theorem~4.1]{HirOr}).
The Elliott classification conjecture would imply that
pureness implies $\cZ$-stability.
Without nuclearity,
both reverse implications are false.
It is shown in
\cite[Theorem 2.2]{NiuWng} that tracial $\cZ$-stability, even
tracial rank zero, does not imply $\cZ$-stability.
The
reduced \ca{s} of nonabelian free groups are pure but
certainly not $\cZ$-stable.
(See \cite[Theorem A]{AGKEP} and
\cite[Proposition 5.13]{AnPrTlVt}.)

Coming back to algebras of the form $C(Y, D)$, it is well known that $C(Y, D)$ need not have strict comparison, even when $D$ does.
Positive results are available under additional assumptions, for instance, when $\dim(Y)=1$.
(See \cite[Theorem 2.6]{AnBoPe}.)
However, recent work in~\cite{SethVil} shows that, under suitable divisibility assumptions on $D$, one can obtain strict comparison for $C(Y, D)$, and in fact establish pureness for arbitrary compact metric spaces $Y$.
(See  \cite[Corollary 5.9 and Theorem 5.11]{SethVil}.)

In this paper, we use these results on pureness of algebras
of the form $C (Y, D)$ to prove that,
if $D$  is pure and $h \colon X \to X$ is a \mh{} (even with
$\mdim (h) > 0$), then $C^* (\Z, C (X, D), h)$
is pure.
Since many reduced free products are pure, in
particular including $C^*_{\mathrm{r}} (F_n)$ when
$n \geq 2$ (see \cite[Theorem A]{AGKEP} and
\cite[Proposition 5.13]{AnPrTlVt}), and simple stably finite
pure \ca{s} have stable rank one (\cite[Corallary 1.3]{HLin}), we are
able to deduce pureness, in particular, strict comparison of
positive elements, as well as stable rank one, for much
more natural versions of some of the examples in
\cite[Section~6]{ArBcPh2}.
In some cases we are also
able to deduce real rank zero.

Related methods also give results for two other kinds of crossed products.
The first arises from actions of compact groups with finite Rokhlin dimension with commuting towers, and the second from actions with the restricted tracial Rokhlin property with comparison.
Actions of the first type were introduced and studied by Gardella  (\cite{Gard17}, \cite{gardella17}), generalizing the notion introduced by Hirshberg, Winter, and Zacharias for finite group actions.
Crossed products by such actions enjoy strong permanence properties:
several important classes of C*-algebras, including $\mathcal{D}$-absorbing C*-algebras (where $\mathcal{D}$ is strongly self-absorbing), stable C*-algebras, C*-algebras with finite nuclear dimension (or decomposition rank), and C*-algebras with finite stable rank (or real rank), are preserved.
(See~\cite[Theorem~4.17]{GdlHrbStg}.)
The second class was introduced in~\cite{MohPhil2021TRP} by the third and fourth named authors, where similar preservation results are obtained.
It is proved there that fixed point algebras under such actions,
and, in appropriate cases, the corresponding crossed products,
preserve properties such as simplicity, Property~(SP), tracial rank zero, tracial rank at most one, the Popa property, tracial $\mathcal{Z}$-stability, $\mathcal{Z}$-stability in the nuclear case, infiniteness, and pure infiniteness.
In~\cite{MohPhil2021TRP2} the authors construct three types of examples that have the restricted tracial Rokhlin property with comparison but do not have finite Rokhlin dimension with commuting towers.

We show that pureness is likewise preserved under crossed products in these settings.
Let $A$ be a simple separable stably finite unital \ca,
let $G$ be a second countable compact group,
and let $\af \colon G \to \Aut (A)$ be an action of $G$ on~$A$.
Suppose that $\af$ has
the restricted tracial Rokhlin property with comparison,
in the sense of \cite[Definition~2.1]{MohPhil2021TRP},
or that $G$ is a Lie group
and $\af$ has finite Rokhlin dimension with commuting towers,
in the sense of \cite[Definition~3.2]{gardella17}.
We prove that if $A$ is pure, then so is $C^* (G, A, \af)$.

The main step
for pureness of $C^* (\Z, C (X, D), h)$ is showing
that if $B$ is a large subalgebra of~$A$, in the
sense of \cite[Definition~4.1]{Phl40}, and if $B$
is pure, then $A$ is pure.
The reverse implication
is also true.
The proof depends on the fact that the
nonprojection parts of the Cuntz semigroups $\Cu (B)$ and $\Cu (A)$ are
isomorphic.
Assuming instead the restricted tracial
Rokhlin property with comparison, or finite Rokhlin
dimension with commuting towers, we compare the
fixed point algebra $A^{\af}$ with~$A$.
The
isomorphism statement above is no longer true,
but enough of it is that the same underlying ideas
can be made to work.
In all three cases, under
suitable additional assumptions, we can say
something about $m$-comparison and $n$-almost
divisibility separately, at the cost of allowing $m$
and~$n$ to increase.
Thus, the fact that
$m$-comparison and $n$-almost divisibility
imply pureness plays a key role.

The paper is organized as follows.
In Section~\ref{S_Prelims},
we establish the relevant notation, definitions, and results
regarding Cuntz subequivalence, the Cuntz semigroup, and
$n$-comparison and $m$-almost divisibility of both
semigroups and associated C*-algebras.
In particular, we
define the key property of pureness for a C*-algebra.
In
Section~\ref{Sec_6314_PLarge}, we establish relationships
between $n$-comparison and $m$-almost divisibility for
a large subalgebra and its containing algebra, and use these
to prove that if a C*-algebra contains a pure large subalgebra,
then the containing algebra is itself pure.
In Section~\ref{Sec_PureCrossProd}, we show that
$C^* \bigl( \Z, \, C (X, D), \, \af \bigr)$ is pure
when $D$ is pure and the action of $\af$ lies over a
minimal homeomorphism.
It follows that such
crossed products have stable rank one and, in certain cases,
real rank zero.
In Section~\ref{Sec_6326_FinRD}, we
establish pureness for
$C^* \bigl( G, \, A, \, \af \bigr)$ when $G$ is a compact
Lie group and $\af$ has finite Rokhlin dimension with
commuting towers.
In Section~\ref{Sec_6326_RTRP}, we
establish pureness for
$C^* \bigl( G, \, A, \, \af \bigr)$ where $G$ is a
second countable compact
group and $\af$ has the tracial Rokhlin property with
comparison.
Finally, in Section~\ref{Sec_Examples} we give
examples of crossed product C*-algebras of the type described
in Section~\ref{Sec_PureCrossProd}.
These examples are not
covered by the results and examples in \cite{ArBcPh2} due
to either the algebra $D$ in $C(X,D)$ not being $\JS$-stable,
the space $X$ not being finite dimensional, or both.
In some of our examples the underlying dynamical system need
not even have mean dimension zero.

\textbf{Acknowledgments}.
The fourth author is grateful to Hannes Thiel
for a stimulating discussion of pure \ca{s}, on the train on the
way back from the conference in Jaca (Spain)
in the summer of 2025.

\section{Preliminaries}\label{S_Prelims}

\subsection{Cuntz Subequivalence, Comparison, and
 Divisibility}\label{SubSec_CuA}

We recall some basic definitions and terminology related to
Cuntz subequivalence and the Cuntz semigroups $\W (A)$ and
$\Cu (A)$ for a C*-algebra $A$.
We include the algebra $A$ in the notation $a \precsim_A b$
and $\ccA{a}$,
because we frequently work simultaneously with both $A$
and a subalgebra $B \subseteq A$,
usually a large subalgebra in the sense of~\cite{Phl40}
or the fixed point algebra of an action of a compact group on~$A$.

\begin{dfn}\label{D_Cuntz_Subeq}
Let $A$ be a C*-algebra.
\begin{enumerate}
\item
For $a,b \in (A \otimes K)_+$, we say that $a$ is
\emph{Cuntz subequivalent} to $b$ over $A$, written
$a \cuntz_A b$, if there is a sequence $(v_{n})_{n=1}^{\infty}$
in $A \otimes K$ such that $\lim_{n \to \infty} v_n b v_n^* = a$.
\item
We say that $a$ and $b$ are \emph{Cuntz equivalent} in $A$,
written $a \sim_{A} b$, if $a \cuntz_{A} b$ and $b \cuntz_{A} a$.
This relation is an equivalence relation, and we write $\cc{a}_A$
for the equivalence class of $a$ in $A$.
When there is no
possibility of confusion, we sometimes just write
$a \cuntz b$, $a \sim b$, and $\cc{a}$.
\item
The \emph{Cuntz semigroup} of $A$ is
\[
\Cu (A) = (A \otimes K)_{+} / \sim_{A},
\]
together with the semigroup operation $\cc{a}_A + \cc{b}_A =
\cc{a \oplus b}_A$ and the partial order given by
$\cc{a}_A \leq \cc{b}_A$ if and only if $a \cuntz_{A} b$.
\item
We define a subsemigroup $\W (A)$ of $\Cu (A)$ by
\[
\W (A) = M_{\infty}(A)_{+} / \sim_{A},
\]
which is the ``classical'' Cuntz semigroup of $A$.
We also denote the class of $a$ in $\W (A)$ by~$\ccA{a}$,
except in Remark~\ref{1903_rem}
and Section~\ref{SubSec_WA}, where we use~$[a]_A$.
\end{enumerate}
\end{dfn}

\begin{rmk}\label{1903_rem}
The operations and order in $\W (A)$ are the same as in $\Cu (A)$.
By Remark~1.2 of \cite{Phl40},
if $a,b \in M_{\infty}(A)_+$, then $\cc{a}_A \leq \cc{b}_A$ in $\Cu (A)$
if and only if $[a]_A \leq [b]_A$ in $\W (A)$.
The notational distinction between $\cc{a}_A$ and $[a]_A$ will only be
used in Section~\ref{SubSec_WA}.
\end{rmk}

The semigroup $\Cu (A)$ tends to have better properties than
$\W (A)$, such as the existence of certain suprema
(Theorem~4.19 of~\cite{APT}), which will be important for
Definition~\ref{D_Cu_Semigp}, and behavior with respect to
direct limits (Theorem~4.35 of~\cite{APT}).
We work in $\Cu (A)$
when possible, but use $\W (A)$ when certain results in the
literature are formulated in those terms.

We next recall the definitions of divisibility and comparison definitions for
abstract semigroups, with the cases of interest being
$\Cu (A)$ and $\W (A)$.
Much of the notation and terminology
follows \cite{AnPrTlVt}, although many of the definitions used
are not original there.

\begin{dfn}[Before Proposition 2.3 of \cite{AnPrTlVt}]\label{LeqS}
For an ordered semigroup $S$ and $\xi,\eta \in S$, write
$\xi \les \eta$ if $(k + 1)\xi \leq k \eta$ for some $k \in \N$.
\end{dfn}

\begin{dfn}\label{nCompDefn}
An ordered semigroup $S$ has \emph{$n$-comparison} if
given $\xi, \eta_{0}, \ldots, \eta_{n} \in S$,
the relation $\xi \les \eta_{j}$
for $j = 0, 1, \ldots, n$ implies $\xi \leq \sum_{j=0}^{n} \eta_{j}$.
\end{dfn}

\begin{dfn}\label{nCompA}
Let $A$ be a C*-algebra.
We say $A$ has \emph{$n$-comparison}
if $\Cu (A)$ has $n$-comparison in the sense of Definition
\ref{nCompDefn}.
If $A$ has 0-comparison, we say that $A$ has
{\emph{strict comparison}}.
\end{dfn}

Strict comparison as in Definition~\ref{nCompA} agrees
with the usual notion of strict comparison of positive elements
by quasitraces for a C*-algebra $A$, and this property holds
if and only if $\Cu (A)$ is almost unperforated.
For details of
these relationships (and relationships with other version of
comparison), see Section~3.4 of~\cite{AnPrTlVt}.

We caution that $n$-comparison in the sense of
Definition~\ref{nCompA} is
different from $r$-comparison (for $r \in [0,\infty)$)
as used in the definition of the radius of comparison, in the sense
of Definition 6.1 of~\cite{Toms06}.

\begin{dfn}\label{D_Cu_WayBelow}
Let $S$ be an ordered semigroup, and let $\xi, \xi' \in S$.
We write $\xi' \ll \xi$ if for every increasing sequence
$(\xi_n)$ in $S$ satisfying $\xi \leq \sup_n \xi_n$, there
is $N \in \N$ such that $\xi' \leq \xi_N$.
We call $\ll$
the \emph{way below} relation (or \emph{compact containment}
relation) and say \emph{$\xi'$ is way below $\xi$}
(or \emph{$\xi'$ is compactly contained in $\xi$}).
\end{dfn}

For this definition to be useful, we need to know that
suprema of increasing sequences in the ordered semigroup
$S$ actually exist.
For the main semigroup of interest here,
this is in fact the case.

\begin{dfn}\label{D_Cu_Semigp}
An ordered semigroup $S$ is said to be a \emph{$\Cu$-semigroup}
if the following properties hold:
\begin{enumerate}
\item
If $(\xi_n)$ is an increasing sequence in $S$,
then $\sup_n \xi_n$ exists.
\item
For any $\xi \in S$ there exists a sequence $(\xi_n)$
such that $\xi_n \ll \xi_{n+1}$ for all $n$ and
$\xi = \sup_n \xi_n$.
\item
If $\xi_1 \ll \xi_2$ and $\eta_1 \ll \eta_2$ in $S$,
then $\xi_1 + \eta_1 \ll \xi_2 + \eta_2$.
\item
If $(\xi_n)$ and $(\eta_n)$ are increasing sequences
in $S$, then $\sup_n (\xi_n + \eta_n) = \sup_n \xi_n + \sup_n \eta_n$.
\end{enumerate}
\end{dfn}

The Cuntz semigroup $\Cu (A)$ is a $\Cu$-semigroup
for a C*-algebra $A$, by the theorem at the start of Section~2
of~\cite{CoElIv2008}.

\begin{dfn}\label{def_almost_divisibility}
Let $S$ be a $\Cu$-semigroup, and let $m \in \N$.
We say that $S$ is \emph{$m$-almost divisible} if
for any $\xi, \xi' \in S$ such that $\xi' \ll \xi$
and any $N \in \N$, there exists $\eta \in S$ such that
\[
N \eta \leq \xi \quad{\mbox{and}}\quad \xi' \leq
(N + 1) (m + 1) \eta.
\]
We say that $S$ is \emph{almost divisible} if $S$ is
$0$-almost divisible.
\end{dfn}

Other closely related notions of almost divisibility exist.
See the discussion after Definition~4.1 of~\cite{AnPrTlVt}
for details.

Applying Definition~\ref{def_almost_divisibility}
to $S = \Cu (A)$ for a C*-algebra $A$ generalizes
Definition~2.5(i) of~\cite{Wint2012}, as in Section~2.3 of ~\cite{RbTk}.

\begin{dfn}\label{D_AlmostDiv_Alg}
Let $A$ be a C*-algebra, and let $m \in \N$.
We say that
$A$ is \emph{$m$-almost divisible} if $\Cu(A)$ is $m$-almost
divisible in the sense of
Definition~\ref{def_almost_divisibility}.
We say that $A$ is
\emph{almost divisible} if $A$ is 0-almost divisible.
\end{dfn}

The following definition first appeared in
Definition~2.6 of~\cite{Wint2012}, with the additional
assumptions the C*-algebra involved is separable, simple, and
unital.
In Section~2.3 of ~\cite{RbTk} it is generalized to
arbitrary C*-algebras, and in
Definition~5.1 of~\cite{AnPrTlVt} to $\Cu$-semigroups.

\begin{dfn}\label{D_PureAlg}
Let $A$ be a C*-algebra and let $m,n \in \N$.
We say that
$A$ is \emph{$(n,m)$-pure} if has $n$-comparison in the
sense of Definition~\ref{nCompA} and is $m$-almost divisible
in the sense of Definition~\ref{def_almost_divisibility}.
We say that $A$ is \emph{pure} if it is $(0,0)$-pure.
\end{dfn}

Due to its importance in many of our results, we state explicitly
the following theorem of \cite{AnPrTlVt}.

\begin{thm}[Theorem 5.7 of \cite{AnPrTlVt}]\label{T_PureEquiv}
Given a C*-algebra $A$, the following are equivalent:
\begin{enumerate}
\item
$A$ is pure; that is, $A$ has strict comparison and is almost divisible.
\item
$A$ is $(n,m)$-pure for some $m,n \in \N$.
\end{enumerate}
\end{thm}

The following lemma will be used repeatedly.
It is Lemma~2.1 of~\cite{Phl40}, but the original source is~\cite{AkmnShlz}.
We include the argument from the proof in~\cite{Phl40}.

\begin{lem}\label{L_3618_Sp01}
Let $A$ be a simple \ca{} which is not of Type~I.
Then there exists $a \in A_{+}$ such that
$\spec (a) = [0, 1]$.
\end{lem}

\begin{proof}
The discussion before~(1) on page~61 of~\cite{AkmnShlz}
shows that $A$ is not scattered in the sense of~\cite{AkmnShlz}.
The conclusion therefore follows
from the argument in~(4) on page~61 of \cite{AkmnShlz}.
\end{proof}

We will frequently need the next lemma,
and a refined version proved afterwards.

\begin{lem}[Lemma 3.6 of \cite{Phl40}]\label{L_3618_Lg36}
Let $A$ be a simple stably finite unital C*-algebra which
is not of Type I.
Let $p \in A \ten K$ be a nonzero
projection, let $N \in \N$, and let $\xi \in \Cu (A) \setminus \set{0}$.
Then there exist $\mu, \kp \in \W (A)$ which are not the classes of \pj{s}
and such that $\mu \leq \cc{p} \leq \mu + \kp$ and $N \kp \leq \xi$.
\end{lem}

\begin{lem}\label{L_5Z26_InterpPj}
Let $A$ be a simple stably finite unital \ca{} which is not of
Type~I.
Let $p \in A \otimes K$ be a \nzp,
let $N \in \N$,
and let $\xi \in \Cu (A) \setminus \{ 0 \}$.
Then there exist $x, y \in (p (A \otimes K) p)_{+}$ and
$\ep \in (0, 1)$ such that:
\begin{enumerate}
\item\label{I_L_5Z26_InnPj_Spec}
$\spec (x) = \spec (y) = [0, 1]$.
\item\label{I_L_5Z26_InnPj_xlp}
$x \precsim_A p$.
\item\label{I_L_5Z26_InnPj_plxy}
$p \precsim_A (x - \ep)_{+} \oplus (y - \ep)_{+}$.
\item\label{I_L_5Z26_InnPj_nylxi}
$N \cc{y}_A \leq \xi$.
\item\label{I_L_5Z26_InnPj_NyLeBe}
$N \cc{y}_A \leq \cc{(x - \ep)_{+}}_A$.
\end{enumerate}
\end{lem}

\begin{proof}
Choose $a \in (A \otimes K)_{+} \setminus \{ 0 \}$
such that $\cc{a}_A = \xi$.
Use Lemma 2.4 of~\cite{Phl40} to choose
orthogonal nonzero positive elements
$b_0, b_1, b_2, \ldots, b_N \in p (A \otimes K) p$
such that $b_0 \sim_A b_1 \sim_A b_2 \sim_A \cdots \sim_A b_N$
and orthogonal nonzero positive elements
$c_1, c_2, c_3 \ldots, c_N \in {\overline{a (A \otimes K) a}}$
such that $c_1 \sim_A c_2 \sim_A \cdots \sim_A c_N$.
By Lemma 2.6 of~\cite{Phl40}, there is
$d \in
  \bigl( {\overline{b_0 (A \otimes K) b_0}} \bigr)_{+}
  \setminus \set{0}$
such that $d \precsim_A c_1$.
Use Lemma~\ref{L_3618_Sp01} to choose
$y \in \bigl( {\overline{d (A \otimes K) d}} \bigr)_{+}$
such that $\spec (y) = [0, 1]$.
Define $x = p - y$ and $\ep = \frac{1}{3}$.
We show that these choices satisfy the conclusion of the lemma.

Part~(\ref{I_L_5Z26_InnPj_Spec}) is clear.
For~(\ref{I_L_5Z26_InnPj_xlp}), we in fact have $x \leq p$.
For~(\ref{I_L_5Z26_InnPj_nylxi}), we have
\[
y \precsim_A d \precsim_A c_1
\andeqn
N \cc{c_1}_A = \cc{c_1}_A + \cc{c_2}_A + \cdots +
\cc{c_N}_A \leq \cc{a}_A = \xi.
\]
We prove~(\ref{I_L_5Z26_InnPj_plxy}).
Set $r = (x - \ep)_{+} + (y - \ep)_{+} \in p (A \otimes K) p$.
Then
\[
\| r - p \|
\leq \bigl\| (x - \ep)_{+} - x \bigr\|
  + \bigl\| (y - \ep)_{+} - y \bigr\|
\leq \frac{1}{3} + \frac{1}{3}
= \frac{2}{3}.
\]
Therefore $r$ is invertible in $p (A \otimes K) p$.
So
\[
p \sim_A r = (x - \ep)_{+} + (y - \ep)_{+}
\precsim_A (x - \ep)_{+} \oplus (y - \ep)_{+},
\]
as desired.

Finally, we prove~(\ref{I_L_5Z26_InnPj_NyLeBe}).
Define a \cfn{} $f \colon [0, \I) \to [0, 1]$ by
\[
f ( \ld) = \begin{cases}
0 & \hspace*{1em} 0 \leq \ld \leq \frac{1}{3} \\
3 \ld - 1 & \hspace*{1em} \frac{1}{3} < \ld \leq \frac{2}{3}\\
1 & \hspace*{1em} \frac{2}{3} < \ld.
\end{cases}
\]
Then $(x - \ep)_{+} \sim f (x)$.
Since
\[
p - x = y \in {\overline{d (A \otimes K) d}}
\S {\overline{b_0 (A \otimes K) b_0}},
\]
we have $x b_j = b_j$ for $j = 1, 2, \ldots, N$.
Therefore also $f (x) b_j = b_j$ for $j = 1, 2, \ldots, N$.
Now
\[
N \cc{b_0}_A = \cc{b_1}_A + \cc{b_2}_A + \cdots +
\cc{b_N}_A \leq \cc{f (x)}_A = \cc{(x - \ep)_{+}}_A.
\]
Since
$y \in {\overline{b_0 (A \otimes K) b_0}}$,
we have $y \precsim_A b_0$,
and (\ref{I_L_5Z26_InnPj_NyLeBe}) follows.
\end{proof}

\subsection{Comparison and Divisibility in $\W (A)$}\label{SubSec_WA}

In this section, we show that $n$-comparison and
$m$-almost divisibility of $\Cu (A)$ are equivalent to
comparison and divisibility properties in  $\W (A)$.
For the clarity of the reader, we distinguish between the class
of an element $a \in A$ when viewed as in $\W (A)$ or $\Cu (A)$,
as noted in Remark~\ref{1903_rem}.

\begin{prp}\label{CP_6325_WA_Cmp}
Let $A$ be a \ca{} and let $n \in \N$.
Then $\Cu (A)$ has $n$-comparison
if and only if $\W (A)$ has $n$-comparison.
\end{prp}

\begin{proof}
Let us first assume that $\Cu (A)$ has $n$-comparison.
Let $a, a_0, a_1, \hdots, a_n \in M_{\infty}(A)_+$ satisfy
$[a]_A \les [a_j]$ in $\W (A)$ for $j = 0, 1, \hdots, n$.
Thus, there are $k_0, k_1, \hdots, k_n \in \Nz$
such that in $\W (A)$ we have
\begin{equation*}
(k_j + 1) [a]_A \leq k_j [a_j]_A.
\end{equation*}
For $j = 0, 1, \hdots, n$, we have
$(k_j + 1) \cc{a}_A \leq k_j \cc{a_j}_A$ in $\Cu (A)$.
By $n$-comparison for $\Cu (A)$,
\begin{equation*}
\cc{a}_A \leq \sum_{j = 0}^n \cc{a_j}_A.
\end{equation*}
By Remark~\ref{1903_rem}, it follows that in $\W (A)$ we have
\[
[a]_A \leq \sum_{j = 0}^n [a_j]_A,
\]
showing $n$-comparison for $\W (A)$.

Let us now assume that $\W (A)$ has $n$-comparison.
Let $a, a_0, a_1, \hdots, a_n \in (A \otimes K)_+$,
and suppose that there are $k_0, k_1, \hdots, k_n \in \Nz$ such that
\begin{equation*}
(k_j + 1) \cc{a}_A \leq k_j \cc{a_j}_A
\end{equation*}
for $j = 0, 1, \hdots, n$.
Let $\ep > 0$.
Choose $\delta_j > 0$ such that in $\Cu (A)$ we have
\[
(k_j + 1) \cc{(a - \ep)_+}_A  \leq k_j \cc{(a_j- \delta_j)_+}_A.
\]
By Lemma~1.9 of~\cite{Phl40}, there exist
$c, c_0, c_1, \hdots, c_n \in M_{\infty}(A)_+$ such that
\begin{equation}\label{2003_comp_eqn_Prv}
\cc{(a - \ep)_+}_A = \cc{c}_A
\end{equation}
and for $j = 0, 1, \hdots, n$ we have
\begin{equation}\label{2003_comp_eqn}
\cc{(a_j- \delta_j)}_A = \cc{c_j}_A.
\end{equation}
Thus, in $\Cu (A)$ we have
\[
(k_j + 1) \cc{c}_A \leq k_j \cc{c_j}_A
\]
and hence, by Remark~\ref{1903_rem}, in $\W (A)$ we have
\[
(k_j + 1) [c]_A \leq k_j [c_j]_A.
\]
By $n$-comparison for $\W (A)$, we see that in $\W (A)$ we have
\begin{equation}\label{2003_comp_eqn_w}
[c]_A \leq \sum_{j = 0}^n [c_j]_A.
\end{equation}
Using (\ref{2003_comp_eqn_Prv}) at the first step,
(\ref{2003_comp_eqn_w}) and Remark~\ref{1903_rem} at the second step,
and (\ref{2003_comp_eqn}) at the third step,
we get the following inequalities in $\Cu (A)$:
\[
\cc{(a - \ep)_+}_A = \cc{c}_A \leq \sum_{j = 0}^n \cc{c_j}_A
 = \sum_{j = 0}^n \cc{(a_j- \delta_j)}_A \leq \sum_{j = 0}^n \cc{a_j}_A.
\]
As $\ep > 0$ is arbitrary, we have
$\cc{a}_A \leq \sum_{j = 0}^n \cc{a_j}_A$,
proving $n$-comparison for $\Cu (A)$.
\end{proof}

Next, we show that $m$-almost divisibility in $\Cu (A)$
is equivalent to a divisibility condition in $\W (A)$.
We recall the following standard result.

\begin{lem}[Proposition~4.3 of~\cite{gardella}]\label{L_5Z28_CptCont}
Let $A$ be a \ca{} and let $a, b \in (A \otimes K)_{+}$.
Then $\cc{a}_A \ll \cc{b}_A$ \ifo{} there is an $\ep > 0$ such
that $a \precsim_A (b - \ep)_{+}$.
\end{lem}

\begin{prp}\label{P_6324_WDiv}
Let $A$ be a \ca{} and let $m \in \N$.
Then $\Cu (A)$ is $m$-almost divisible if and only if
for every $a \in M_{\infty} (A)_+$, every $k \in \Nz$, and every $\ep > 0$,
there exists $b \in M_{\infty} (A)_+$ such that in $\W (A)$ we have
\[
k [b]_A \leq [a]_A
\andeqn
[(a - \ep)_+]_A \leq (k + 1) (m + 1) [b]_A.
\]
\end{prp}

\begin{proof}
Let $\Cu (A)$ be $m$-almost divisible.
Let $a \in M_{\infty}(A)_+$ and let $k \in \Nz$.
Let $\ep > 0$ be arbitrary.
Since
$\inn{\left( a - \frac{2 \ep}{3} \right)_+}_A
 \ll \inn{\left(a - \frac{\ep}{3} \right)_+}_A$,
and $\Cu (A)$ is $m$-almost divisible,
there exists $b \in (A \otimes K)_+$ such that
\begin{equation}\label{1903_div_wtc}
k \cc{b}_A \leq \inn{\left(a - \frac{\ep}{3} \right)_+}_A
\quad {\mbox{and}} \quad
\inn{\left(a - \frac{2 \ep}{3} \right)_+}_A \leq (k + 1) (m + 1) \cc{b}_A.
\end{equation}
Choose $\delta > 0$ such that
\begin{equation}\label{1903_cutdown_wtc}
\inn{\left(a - \ep\right)_+}_A \leq (k + 1) (m + 1) \cc{(b- \delta)_+}_A.
\end{equation}
By Lemma~1.9 of~\cite{Phl40},
there exists $c\in M_{\infty}(A)_+$ such that
\[
\cc{(b- \delta)_+}_A = \cc{c}_A.
\]
Hence, by (\ref{1903_div_wtc}) and (\ref{1903_cutdown_wtc}), we see that
\begin{equation*}
k \cc{c}_A \leq \cc{a}_A
\andeqn
\cc{(a - \ep)_+}_A \leq (k + 1) (m + 1) \cc{c}_A.
\end{equation*}
Finally, by Remark~\ref{1903_rem}, it follows that in $\W (A)$ we have
\[
k [c]_A \leq [a]_A
\andeqn
[(a - \ep)_+] \leq (k + 1) (m + 1) [c]_A.
\]
This is the desired relation.

Conversely, assume that for all $\ep > 0$,
$k \in \Nz$, and $a \in M_{\infty}(A)_+$,
there exists $b \in M_{\infty}(A)_+$ such that in $\W (A)$ we have
\[
k[b]_A \leq [a]_A
\andeqn
[(a - \ep)_+] \leq (k + 1) (m + 1) [b]_A.
\]
Let $k \in \Nz$ and let $\xi, \eta \in \Cu (A)$ satisfy $\xi \ll \eta$.
Choose $a \in (A \otimes K)_+$ such that
\begin{equation}\label{Eq_6327_eta}
\cc{a}_A = \eta.
\end{equation}
By Lemma~\ref{L_5Z28_CptCont}, there is $\ep > 0$ such that
\begin{equation}\label{1903_almost_div_ctw}
\xi \leq \cc{(a - \ep)_+}_A.
\end{equation}
By Lemma~1.9 of~\cite{Phl40},
there exists $ c\in M_{\infty}(A)_+$ such that
\begin{equation*}
\cc{c}_A = \inn{\left(a - \frac{\ep}{2} \right)_+}_A.
\end{equation*}
Choose $\delta > 0$ such that
\begin{equation}\label{Eq_6327_NN}
\left(a - \ep\right)_+ \precsim_A (c- \delta)_+.
\end{equation}
By hypothesis, there exists $b \in M_{\infty}(A)_+$ such that
in $\W (A)$ we have
\[
k[b]_A \leq [c]_A
\andeqn
\left[\left(c- \delta\right)_+\right]_A \leq (k + 1) (m + 1) [b]_A.
\]
In $\Cu (A)$ we now get,
using Remark~\ref{1903_rem} and~(\ref{Eq_6327_eta}),
\begin{equation*}
k \cc{b}_A \leq \cc{c}_A
 = \inn{\left(a - \frac{\ep}{2} \right)_+}_A \leq \eta
\end{equation*}
and using (\ref{1903_almost_div_ctw}), (\ref{Eq_6327_NN}),
and Remark~\ref{1903_rem},
\begin{equation*}
\xi \leq \cc{(a - \ep)_+}_A
 \leq  \inn{\left(c- \delta\right)_+}_A \leq (k + 1) (m + 1) \cc{c}_A.
\end{equation*}
This proves $m$-almost divisibility for $\Cu (A)$.
\end{proof}

\section{Pureness for Large Subalgebras}\label{Sec_6314_PLarge}

In this section, we prove that if $B \S A$ is a large subalgebra
the sense of Definition~4.1 of~\cite{Phl40},
then $A$ is pure \ifo{} $B$ is pure.
We obtain results for comparison and almost divisibility
independently, with an increase of the parameter in the case
of almost divisibility.
For our applications, we do not need to know that pureness of $A$
implies pureness of~$B$.

\subsection{Preliminary Results}\label{Sec_6314_Comp}
We start by collecting material needed for both
comparison and almost divisibility.

\begin{ntn}\label{N_6315_Cup}
Let $A$ be a stably finite simple \uca.
Following Definition~3.1 of~\cite{Phl40},
we let $\Cu_{+} (A)$ denote the set of $\et \in \Cu (A)$ such that
$\et$ is not the class of a \pj;
that is, $\et$ is purely positive in the sense of~\cite{APT}.
See also \cite{EllRoSa}.
\end{ntn}

\begin{thm}[Theorem 6.8 of \cite{Phl40}]\label{T_6318_Isom_CuPl}
Let $A$ be an infinite dimensional stably finite simple
unital C*-algebra, and let $B \subseteq A$ be a subalgebra
which is large in the sense of Definition~4.1 of~\cite{Phl40}.
Let $\io \colon B \to A$ be the inclusion map.
Then $\io_{*}$ defines an order and semigroup isomorphism
from $\Cu_{+}(B) \cup \set{0}$ to $\Cu_{+}(A) \cup \set{0}$.
\end{thm}

The following corollary is immediate.

\begin{cor}\label{C_6318_LgCmp}
Let the assumptions be as in Theorem~\ref{T_6318_Isom_CuPl}.
Then for any $x \in (A \otimes K)_{+}$ for which
$\ccA{x}$ is not the class of a projection in $\Cu (A)$,
there is a $y \in (B \otimes K)_{+}$ such that $x \sim_{A} y$
and $\ccB{y}$ is not the class of projection in $\Cu (B)$.
\end{cor}

The appropriate specialization of Lemma~\ref{L_3618_Sp01} is as follows.

\begin{lem}\label{L_3618_LgSp01}
Let $A$ be an infinite dimensional stably finite simple
unital C*-algebra, let $B \subseteq A$ be a subalgebra
which is large in the sense of Definition~4.1 of~\cite{Phl40}
(possibly $B = A$),
and let $C \S B$ be a nonzero \hsa.
Then there exists $a \in C_{+}$ such that
$\spec (a) = [0, 1]$.
\end{lem}

\begin{proof}
The algebra $B$ is simple by Proposition~5.2 of~\cite{Phl40},
infinite dimensional by Proposition~5.5 of~\cite{Phl40},
and unital.
Therefore $B$ is not of Type~I, so $C$ is also is not of Type~I.
Thus Lemma~\ref{L_3618_Sp01} applies.
\end{proof}

\subsection{Comparison for Large Subalgebras}\label{Sec_6314_Comp}

In this subsection, we show that a
large subalgebra has $m$-comparison if and only if
the containing algebra does.

\begin{thm}\label{T_6314_CompThm}
Let $A$ be a simple separable infinite dimensional stably
finite unital C*-algebra, let $B \S A$ be a subalgebra
which is large in the sense of Definition~4.1 of~\cite{Phl40},
and let $m \in \N$.
Then $B$ has $m$-comparison if and only if $A$ has $m$-comparison.
\end{thm}

\begin{proof}
We begin by showing that $m$-comparison for $\Cu (B)$ implies
$m$-comparison for $\Cu (A)$.
Let $x, a_{0}, a_{1}, \ldots, a_{m} \in (A \ten K)_{+}$
satisfy $\cc{x}_A \les \cc{a_{j}}_A$ for
$j = 0, 1, \ldots, m$.
We must show that
\begin{equation}\label{Eq_6315_Cmp_to_pr}
\cc{x}_A \leq \sum_{j = 0}^m \cc{a_j}_A.
\end{equation}
If $\cc{x}_A = 0$, this is immediate, so we assume
that $\cc{x}_A \neq 0$.
By the definition of $\les$ (Definition~\ref{LeqS}),
for $j = 0, 1, \ldots, m$ there is $k_j \in \N$
such that
\begin{equation}\label{Eq_6318_M_SubS}
(k_j + 1) \ccA{x} \leq k_j \ccA{a_j}.
\end{equation}

First assume $\cc{x}_A$ is not the class of a projection.
Use Corollary~\ref{C_6318_LgCmp}
to choose $y \in (B \otimes K)_{+}$ such that $\ccA{y} = \ccA{x}$
and $0$ is a limit point of $\spec (y)$.
Let $\varepsilon > 0$ be arbitrary. For $j = 0, 1, \ldots, m$,
we will find $b_j \in (B \otimes K)_{+}$ such that
\begin{equation}\label{Eq_6318_L}
\ccB{(y - \ep)_{+}} \les \ccB{b_j}
\andeqn
\ccA{b_j} \leq \ccA{a_j}.
\end{equation}
Since $0$ is a limit point of $\spec (y)$,
there is $\rh \in (0, \ep)$ such that
$\spec (y) \cap (\rh, \ep) \neq \E$.
Choose $\ld \in \spec (y) \cap (\rh, \ep)$.
Choose a \cfn{} $f \colon [0, \I) \to [0, 1]$ such that
$\supp (f) \S (\rh, \ep)$ and $f (\ld) \neq 0$.

Define
\[
I = \bigl\{ j \in \{0, 1, \ldots, m\} \colon
 {\mbox{$\cc{a_j}_A$ is the class of a nonzero projection}} \bigr\}.
\]

For $j \not\in I$, Corollary~\ref{C_6318_LgCmp} provides
$b_j \in (B \ten K)_{+}$ such that
$\cc{b_j}_A = \cc{a_j}_A$ in $\Cu (A)$.
Use Lemma~\ref{L_3618_LgSp01} to choose
$c \in \bigl( {\overline{f (y) (B \otimes K) f (y)}} \bigr)_{+}$
such that $\spec (c) = [0, 1]$.
Set $z_j = (y - \ep)_{+} + c$.
Then $(y - \ep)_{+} \precsim_B z_j \precsim_B y$.
We have
\[
(k_j + 1) \ccA{z_j}
 \leq (k_j + 1) \ccA{y}
 = (k_j + 1) \ccA{x}
 \leq k_j \ccA{a_j}
 = k_j \ccA{b_j}.
\]
Since $0$ is a limit point of both $\spec (z_j)$ and $\spec (b_j)$,
Theorem~\ref{T_6318_Isom_CuPl} gives the second step in
the calculation
\[
(k_j + 1) \ccB{(y - \ep)_{+}}
 \leq (k_j + 1) \ccB{z_j}
 \leq k_j \ccB{b_j}.
\]
It follows that $\ccB{(y - \ep)_{+}} \les \ccB{b_j}$,
and we already have $\cc{b_j}_A = \cc{a_j}_A$.
So~(\ref{Eq_6318_L}) holds.

Now suppose $j \in I$.
We may as well assume that $a_j$ is a \pj.
Use Lemma~\ref{L_5Z26_InterpPj} to choose
$r_j, s_j \in (a_j (A \otimes K) a_j)_{+}$ and $\dt \in (0, 1)$
such that
\[
\spec (r_j) = \spec (s_j) = [0, 1],
\qquad
r_j \precsim_A a_j,
\]
\[
a_j \precsim_A (r_j - \dt)_{+} \oplus (s_j - \dt)_{+},
\qquad {\mbox{and}} \qquad
2 k_j \ccA{s_j} \leq \ccA{(r_j - \dt)_{+}}.
\]
Corollary~\ref{C_6318_LgCmp} yields $b_j \in (B \otimes K)_{+}$
such that $0$ is a limit point of $\spec (b_j)$ and $\ccA{b_j} = \ccA{r_j}$.
Use Lemma~2.6 of~\cite{Phl40} to choose
$d_j \in \bigl( {\overline{f (y) (A \otimes K) f (y)}} \bigr)_{+}
    \setminus \{ 0 \}$
such that $d_j \precsim_A s_j$.
Use Lemma~\ref{L_3618_LgSp01} to choose
$g_j \in \bigl( {\overline{d_j (A \otimes K) d_j}} \bigr)_{+}$
such that $\spec (g_j) = [0, 1]$.
Corollary~\ref{C_6318_LgCmp} yields $c_j \in (\Kt B)_{+}$
such that $0$ is a limit point of $\spec (c_j)$ and $\ccA{c_j} = \ccA{g_j}$.
Set $z_j = (y - \ep)_{+} \oplus c_j$.
We have $c_j \precsim_A f (y)$, so
\[
\begin{split}
(2 k_j + 2) \ccA{z_j}
& \leq (2 k_j + 2) \ccA{y}
  = (2 k_j + 2) \ccA{x}
  \leq 2 k_j \ccA{a_j}
\\
& \leq 2 k_j \ccA{(r_j - \dt)_{+}} + 2 k_j \ccA{(s_j - \dt)_{+}}
\\
& \leq (2 k_j + 1) \ccA{r_j}
  = (2 k_j + 1) \ccA{b_j}.
\end{split}
\]
Since $0$ is a limit point of both $\spec (z_j)$ and $\spec (b_j)$,
Theorem~\ref{T_6318_Isom_CuPl} gives the second step in
the calculation
\[
(2 k_j + 2) \ccB{(y - \ep)_{+}}
 \leq (2 k_j + 2) \ccB{z_j}
 \leq (2 k_j + 1) \ccB{b_j}.
\]
It follows that $\ccB{(y - \ep)_{+}} \les \ccB{b_j}$.
We already have $\ccA{b_j} = \ccA{r_j} \leq \ccA{a_j}$.
So~(\ref{Eq_6318_L}) holds in this case too.

Since (\ref{Eq_6318_L}) holds for $j = 0, 1, \ldots, m$
and $B$ has $m$-comparison, it follows that
\[
\ccB{(y - \ep)_{+}} \leq \sum_{j = 0}^m \ccB{b_j}.
\]
Using the second part of~(\ref{Eq_6318_L}), we get
\[
\ccA{(y - \ep)_{+}}
 \leq \sum_{j = 0}^m \ccA{b_j}
 \leq \sum_{j = 0}^m \ccA{a_j}.
\]
Since $\ep > 0$ is arbitrary, we get
\[
\ccA{y} \leq \sum_{j = 0}^m \ccA{a_j}.
\]
Combining this with $\ccA{x} = \ccA{y}$ gives~(\ref{Eq_6315_Cmp_to_pr})
and completes the proof of this direction in the case that $\ccA{x}$
is not the class of a \pj.

Now assume that $\cc{x}_A$ is the class of a nonzero
projection in $A \otimes K$.
Since $x \neq 0$, we have $a_j \neq 0$ for $j = 0, 1, \ldots, m$.
Use Lemma 2.6 of \cite{Phl40} to choose
$a \in (A \otimes K)_{+} \setminus \{0\}$ such that
$\cc{a}_A \leq \cc{a_j}_A$ for $j = 1, 2, \hdots, m$.
Set $k = \max ( k_0, k_1, \ldots, k_m )$.
By Lemma~\ref{L_3618_Lg36}, there exist $\mu, \eta \in \Cu_{+} (A)$
such that
\[
\mu \leq \cc{x}_A \leq \mu + \eta
\andeqn
(2 k + 2) \eta \leq \cc{a}_A.
\]
For $j = 0, 1, \hdots, m$, we then have,
also using (\ref{Eq_6318_M_SubS}) at the second step,
\[
\begin{split}
(2 k_j + 2) (\mu + \eta)
& \leq (2 k_j + 2) \cc{x}_A + (2 k_j + 2) \eta
\\
& \leq 2 k_j \cc{a_j}_A + \cc{a}_A
  \leq 2 k_j \cc{a_j}_A + \cc{a_j}_A
  = (2 k_j + 1) \cc{a_j}_A.
\end{split}
\]
Hence $\mu + \eta \les \cc{a_j}_A$.
Since $\mu + \eta \in \Cu_{+} (A)$, the previous case yields
the second step in the calculation
\[
\cc{x}_A \leq \mu + \eta \leq \sum_{j = 0}^m \cc{a_j}_A,
\]
which is~(\ref{Eq_6315_Cmp_to_pr}).
This completes the proof that if $\Cu (B)$ has
$m$-comparison, then so does $\Cu (A)$.

The reverse implication uses the same pieces, but arranged differently.
Assume that $\Cu (A)$ has $m$-comparison;
we show that $\Cu (B)$ has $m$-comparison as well.
To this end, let
$y, b_0, b_1, \ldots, b_m \in (B \otimes K)_{+}$ satisfy
$\cc{y}_B \les \cc{b_j}_B$ for $j = 0, 1, \ldots, m$.
\Wolog{} $y \neq 0$.
By the definition of $\les$ (Definition~\ref{LeqS}),
for $j = 0, 1, \ldots, m$ there is $k_j \in \N$
such that $(k_j + 1) \ccB{y} \leq k_j \ccB{b_j}$.
In particular, $b_j \neq 0$ for $j = 0, 1, \ldots, m$.

First assume $\cc{y}_B$ is not the class of a projection.
Let $\varepsilon > 0$ be arbitrary. We will find $z \in (B \otimes K)_{+}$
and for $j = 0, 1, \ldots, m$, we will find $a_j \in (B \otimes K)_{+}$
such that $0$ is a limit point of $\spec (z)$,
$(y - \ep)_{+} \precsim_B z$,
and, for $j = 0, 1, \ldots, m$, $\ccA{z} \les \ccA{a_j}$,
$0$ is a limit point of $\spec (a_j)$, and $\ccB{a_j} \leq \ccB{b_j}$.

Since $0$ is a limit point of $\spec (y)$,
there is $\rh \in (0, \ep)$ such that
$\spec (y) \cap (\rh, \ep) \neq \E$.
Choose $\ld \in \spec (y) \cap (\rh, \ep)$.
Choose a \cfn{} $f \colon [0, \I) \to [0, 1]$ such that
$\supp (f) \S (\rh, \ep)$ and $f (\ld) \neq 0$.
Use Lemma~\ref{L_3618_LgSp01} to choose
$c \in \bigl( {\overline{f (y) (B \otimes K) f (y)}} \bigr)_{+}$
such that $\spec (c) = [0, 1]$.
Set $z = (y - \ep)_{+} + c$.
Then $0$ is a limit point of $\spec (z)$ and
$(y - \ep)_{+} \precsim_B z \precsim_B y$.

Define
\[
I = \bigl\{ j \in \{0, 1, \ldots, m\} \colon
 {\mbox{$\cc{b_j}_B$ is the class of a nonzero projection}} \bigr\}.
\]

For $j \not\in I$, set $a_j = b_j$.
Then $0$ is a limit point of $\spec (a_j)$.
Also,
\[
\ccB{z} \leq \ccB{y}
\andeqn
(k_j + 1) \ccB{y} \leq k_j \ccB{a_j}.
\]
The second says that $\ccB{z} \les \ccB{a_j}$.

Now suppose $j \in I$.
We may as well assume that $b_j$ is a \pj.
Use Lemma~\ref{L_5Z26_InterpPj} to choose
$a_j, s_j \in (b_j (B \otimes K) b_j)_{+}$ and $\dt \in (0, 1)$
such that
\[
\spec (a_j) = \spec (s_j) = [0, 1],
\qquad
a_j \precsim_B b_j,
\]
\[
b_j \precsim_B (a_j - \dt)_{+} \oplus (s_j - \dt)_{+},
\qquad {\mbox{and}} \qquad
2 k_j \ccB{s_j} \leq \ccB{(a_j - \dt)_{+}}.
\]
Then
\[
\begin{split}
(2 k_j + 2) \ccB{z}
& \leq (2 k_j + 2) \ccB{y}
\\
& \leq 2 k_j \ccB{b_j}
  \leq 2 k_j \ccB{a_j} + 2 k_j \ccB{s_j}
  \leq (2 k_j + 1) \ccB{a_j}.
\end{split}
\]
Thus, here too $\ccB{z} \les \ccB{a_j}$,
$0$ is a limit point of $\spec (a_j)$,
and $\ccA{a_j} \leq \ccA{b_j}$ because $\ccB{a_j} \leq \ccB{b_j}$.

Passing to~$A$, we have
$\ccA{z} \les \ccA{a_j}$ for $j = 0, 1, \ldots, m$.
Since $A$ has $m$-comparison,
it follows that
\[
\ccA{z} \leq \sum_{j = 0}^m \ccA{a_j}.
\]
Since $0$ is a limit point of $\spec (z)$
and of $\spec (a_j)$ for $j = 0, 1, \ldots, m$,
Theorem~\ref{T_6318_Isom_CuPl} implies that
\[
\ccB{z} \leq \sum_{j = 0}^m \ccB{a_j}.
\]
Therefore
\[
\ccB{(y - \ep)_{+}}
 \leq \ccB{z}
 \leq \sum_{j = 0}^m \ccB{a_j}
 \leq \sum_{j = 0}^m \ccB{b_j}.
\]

We have proved that for all $\ep > 0$ we have
\[
\ccB{(y - \ep)_{+}} \leq \sum_{j = 0}^m \ccB{b_j}.
\]
So $\ccB{y} \leq \sum_{j = 0}^m \ccB{b_j}$.
This completes the case where $\ccB{y}$ is not the class of a \pj.

Now assume that $\cc{y}_B$ is the class of a nonzero
projection in $B \otimes K$.
Use Lemma 2.6 of \cite{Phl40} to choose
$b \in (B \otimes K)_{+} \setminus \{0\}$ such that
$\cc{b}_B \leq \cc{b_j}_B$ for $j = 0, 1, \hdots, m$.
Set $k = \max ( k_0, k_1, \ldots, k_m )$.
By Lemma~\ref{L_3618_Lg36}, there exist $\mu, \eta \in \Cu_{+} (B)$
such that
\[
\mu \leq \cc{y}_B \leq \mu + \eta
\andeqn
(2 k + 2) \eta \leq \cc{b}_B.
\]
For $j = 0, 1, \hdots, m$, we then have
\[
\begin{split}
(2 k_j + 2) (\mu + \eta)
& \leq (2 k_j + 2) \cc{y}_B + (2 k_j + 2) \eta
\\
& \leq 2 k_j \cc{b_j}_B + \cc{b}_B
  \leq 2 k_j \cc{b_j}_B + \cc{b_j}_B
  = (2 k_j + 1) \cc{b_j}_B.
\end{split}
\]
Hence $\mu + \eta \les \cc{b_j}_B$.
Since $\mu + \eta \in \Cu_{+} (B)$, the previous case yields
the second step in the calculation
\[
\cc{y}_B \leq \mu + \eta \leq \sum_{j = 0}^m \cc{a_j}_B.
\]
This completes the proof that if $\Cu (A)$ has
$m$-comparison, then so does $\Cu (B)$.
\end{proof}

\subsection{Almost Divisibility for Large Subalgebras}\label{Sec_6314_AD}
In this subsection, we prove that
if $A$ is a simple stably finite \uca{} and $B \S A$ is a
large subalgebra which is $m$-almost divisible, then $A$ is
$(2 m + 1)$-almost divisible.
In the opposite direction, if $A$ is
$m$-almost divisible, then $B$ is $(4m + 3)$-almost divisible.
This differs from the behavior of $m$-comparison
as proved in Section~\ref{Sec_6314_Comp},
where the parameter $m$ remains unchanged when
passing between $A$ and~$B$.
We don't know whether the stronger version is true for
$m$-almost divisibility.

\begin{lem}\label{L_5Z29_CC_down}
Let $A$ be an infinite dimensional simple stably finite \uca,
and let $B \S A$ be a subalgebra
which is large in the sense of Definition~4.1 of~\cite{Phl40}.
Let $\io \colon B \to A$ be the inclusion map.
Let $\xi, \et \in \Cu (B)$ satisfy $\io_* (\xi) \ll \io_* (\et)$.
Then $\xi \ll \et$.
\end{lem}

\begin{proof}
Choose $x, y \in B \otimes K$ such that such that
$\cc{x}_B = \xi$ and $\cc{y}_B = \et$.
If there is a
\pj{} $p \in B$ such that $\xi = \cc{p}_B$, then
\[
\xi = \cc{p}_B \ll \cc{p}_B = \xi \leq \et,
\]
which is the desired conclusion.
The conclusion follows the
same way if $\et$ is the class of a \pj.
Therefore we may
assume that $0$ is a limit point of both $\spec (x)$ and
$\spec (y)$.

Since $\cc{x}_A \ll \cc{y}_A$, Lemma~\ref{L_5Z28_CptCont}
provides $\ep > 0$ such that $x \precsim_A (y - \ep)_{+}$.
Since $0$ is a limit point of $\spec (y)$, there is
$\rh \in (0, \ep)$ such that
$\spec (y) \cap (\rh, \ep) \neq \E$.
Choose
$\ld \in \spec (y) \cap (\rh, \ep)$.
Choose a \cfn{} $f \colon [0, \I) \to [0, 1]$ such that
$\supp (f) \S (\rh, \ep)$ and $f (\ld) \neq 0$.
Use Lemma~\ref{L_3618_LgSp01} to choose
$c \in \bigl( {\overline{f (y) (B \otimes K) f (y)}} \bigr)_{+}$
such that $\spec (c) = [0, 1]$.
Then
\[
(y - \ep)_{+} \leq (y - \ep)_{+} + c \precsim_B (y - \rh)_{+}.
\]
It follows that
\[
x \precsim_A (y - \ep)_{+} \leq (y - \ep)_{+} + c.
\]
Since $0$ is a limit point of both $\spec (x)$
and $\spec \bigl( (y - \ep)_{+} + c \bigr)$,
Theorem~\ref{T_6318_Isom_CuPl} implies that
$x \precsim_B (y - \ep)_{+} + c$.
Combining this with
$(y - \ep)_{+} + c \precsim_B (y - \rh)_{+}$,
we get $x \precsim_B (y - \rh)_{+}$.
So
$\cc{x}_B \ll \cc{y}_B$ by
Lemma~\ref{L_5Z28_CptCont}.
\end{proof}

\begin{prp}\label{P_5Z28_Large_alm_dv}
Let $A$ be an infinite dimensional simple stably finite \uca,
and let $B \S A$ be a subalgebra which is large in the sense
of Definition~4.1 of~\cite{Phl40}.
Let $m \in \Nz$.
Suppose that $B$ is $m$-almost divisible.
Then $A$ is $(2 m + 1)$-almost divisible.
\end{prp}

\begin{proof}
Let $\io \colon B \to A$ be the inclusion map.
Let $k \in \N$.
Let $\xi, \et \in \Cu (A)$ satisfy $\xi \ll \et$.
Choose $a, b \in (A \otimes K)_{+}$ such that
$\cc{a}_A = \xi$ and $\cc{b}_A = \et$.

First suppose that $a = 0$ or $0$ is a limit point of
$\spec (a)$, and that $b = 0$ or $0$ is a limit point of $\spec (b)$.
We claim that in this case, we
actually get $\gm \in \Cu (A)$ such that
\begin{equation}\label{Eq_5Z28_Case_1}
k \gm \leq \et
\andeqn
\xi \leq (k + 1) (m + 1) \gm.
\end{equation}
This claim certainly implies the desired conclusion for
this case, which has merely $k \gm \leq \et$ and
$\xi \leq (k + 1) (2 m + 2) \gm$.

To prove the claim, use Corollary~\ref{C_6318_LgCmp}
to find $\xi_0, \et_0 \in \Cu (B)$
such that $\io_* (\xi_0) = \xi$ and $\io_* (\et_0) = \et$.
Then $\xi_0 \ll \et_0$ by Lemma~\ref{L_5Z29_CC_down}.
So the hypothesis provides $\gm_0 \in \Cu (B)$
such that $k \gm_0 \leq \et_0$ and
$\xi_0 \leq (k + 1) (m + 1) \gm_0$.
Set $\gm = \io_* (\gm_0)$.
Then~(\ref{Eq_5Z28_Case_1}) holds, as claimed.

Otherwise, one or both of $a$ and $b$ is nonzero but
its spectrum does not have $0$ as a limit point.
Therefore there is a \nzp{} $p \in A \otimes K$
such that $\cc{p}_A = \cc{a}_A$ or there is a \nzp{}
$p \in A \otimes K$ such that $\cc{p}_A = \cc{b}_A$.
In either case, it suffices to prove the desired
conclusion when $a = b = p$.

The case $N = 1$ of Lemma~\ref{L_5Z26_InterpPj} provides
$x, y \in (p (A \otimes K) p)_{+}$ and $\ep \in (0, 1)$
such that
\[
\spec (x) = \spec (y) = [0, 1],
\quad
x \precsim_A p,
\quad
p \precsim_A (x - \ep)_{+} \oplus (y - \ep)_{+},
\quad {\mbox{and}} \quad
y \precsim_A (x - \ep)_{+}.
\]
Lemma~\ref{L_5Z28_CptCont} tells us that
$\cc{(x - \ep)_{+}}_A \ll \cc{x}_A$.
Since $\ep < 1$,
it follows that $0$ is a limit point of
$\spec ((x - \ep)_{+})$ as well as of $\spec (x)$.
By the claim above (see~(\ref{Eq_5Z28_Case_1})),
there is $\gm \in \Cu (A)$ such that
\begin{equation}\label{Eq_5Z28_xme_x}
k \gm \leq \cc{x}_A
\andeqn
\cc{(x - \ep)_{+}}_A \leq (k + 1) (m + 1) \gm.
\end{equation}
Now
\[
k \gm \leq \cc{x}_A \leq \cc{p}_A
\]
and, using $y \precsim_A (x - \ep)_{+}$ at the second
step and the second part of~(\ref{Eq_5Z28_xme_x}) at the
third step,
\[
\cc{p}_A \leq \cc{(x - \ep)_{+}}_A +
\cc{(y - \ep)_{+}}_A
\leq 2 \cc{(x - \ep)_{+}}_A
\leq (k + 1) (2 m + 2) \gm.
\]
This completes the proof of the proposition.
\end{proof}

We now give the the analog of
Proposition~\ref{P_5Z28_Large_alm_dv} in the other direction.

\begin{prp}\label{P_5Z28_Large_alm_dv_C}
Let $A$ be an infinite dimensional simple stably finite \uca,
and let $B \S A$ be a subalgebra which is large in the sense
of Definition~4.1 of~\cite{Phl40}.
Let $m \in \Nz$.
Suppose that $A$ is $m$-almost divisible.
Then $B$ is $(4 m + 3)$-almost divisible.
\end{prp}

\begin{proof}
The proof of this proposition is similar to
that of Proposition~\ref{P_5Z28_Large_alm_dv}.
Let $k \in \N$ and let $\xi, \et \in \Cu (B)$ satisfy
$\xi \ll \et$.
Choose $a, b \in (B \otimes K)_{+}$
such that $\cc{a}_B = \xi$ and $\cc{b}_B = \et$.

First suppose that $a = 0$ or $0$ is a limit point of
$\spec (a)$, and that $b = 0$ or $0$ is a limit point of $\spec (b)$.
In this case, we will prove the following stronger conclusion:
for every $k \in \N$ there is $\gm \in \Cu (B)$ such that
\begin{equation}\label{Eq_6318_Strong}
k \gm \leq \cc{b}_B \quad \text{and} \quad
\cc{a}_B \leq (k + 1) (2 m + 2) \gm.
\end{equation}
Since $\cc{a}_B \ll \cc{b}_B$ and \hm{s} preserve the
relation $\ll$ (this is obvious from Lemma~\ref{L_5Z28_CptCont}),
we have $\cc{a}_A \ll \cc{b}_A$ in $\Cu (A)$.
Using $m$-almost divisibility of $A$, we obtain
$c\in (A \otimes K)_{+}$ such that
\begin{equation}\label{Eq_A_div}
2 k \cc{c}_A \leq \cc{b}_A \quad \text{and} \quad
\cc{a}_A \leq (2 k + 1) (m + 1) \cc{c}_A.
\end{equation}

Suppose that $\cc{c}_A \in \Cu_{+}(A) \cup \set{0}$,
so that $2 \cc{c}_A \in \Cu_{+} (A) \cup \set{0}$.
By Corollary~\ref{C_6318_LgCmp}, there is $d \in (B \otimes K)_{+}$
such that $\ccA{d} = \ccA{c}$ and $0$ is a limit point of $\spec (d)$.
Using (\ref{Eq_A_div}), we have
\[
k \cc{d}_A \leq 2 k \cc{d}_A \leq \cc{b}_A
\]
and
\[
\cc{a}_A \leq (2 k + 1) (m + 1) \cc{d}_A
         \leq (k + 1) (2 m + 2) \cc{d}_A.
\]
By Theorem~\ref{T_6318_Isom_CuPl}, these relations
also hold for the classes in $\Cu (B)$.
This is (\ref{Eq_6318_Strong}) with $\gm = \cc{d}_B$.

Let us now suppose that $\cc{c}_A$ is a class of nonzero projection.
By Lemma~\ref{L_3618_Lg36}, there exist
$\mu, \kp \in \Cu_{+} (A)$ such that
\[
\mu \leq \cc{c}_A \leq \mu + \kp \quad \text{and}
\quad \kp \leq \cc{c}_A.
\]
Then $\mu + \kp \in \Cu_{+} (A)$, and from (\ref{Eq_A_div})
we see that
\[
k (\mu + \kp) \leq 2 k \cc{c}_A \leq \cc{b}_A \quad
\text{and} \quad \cc{a}_A \leq (k + 1) (2 m +2) (\mu + \kp).
\]
Corollary~\ref{C_6318_LgCmp} provides
$d \in (B \otimes K)_{+}$ such that $0$ is a limit point of $\spec(d)$
and $\cc{d}_A = \mu + \kp$, and Theorem~\ref{T_6318_Isom_CuPl} implies that
\[
k \cc{d}_B \leq \cc{b}_B
\andeqn
\cc{a}_B \leq (k + 1) (2 m + 2) \cc{d}_B.
\]
This is (\ref{Eq_6318_Strong}) with $\gm = \cc{d}_B$.

It remains to consider the case in which
one or both of $a$ and $b$ is nonzero but
its spectrum does not have $0$ as a limit point.
In this case, there is a \nzp{} $p \in B \otimes K$
such that $\cc{p}_B = \cc{a}_B$ or there is a \nzp{}
$p \in B \otimes K$ such that $\cc{p}_B = \cc{b}_B$.
In either case, $\ccB{a} \leq \ccB{p} \ll \ccB{p} \leq \ccB{b}$,
and it suffices to prove the desired
conclusion when $a = b = p$.
By Lemma~\ref{L_5Z26_InterpPj}, there exist
$x, y \in (p (B \otimes K) p)_{+}$ and $\varepsilon \in (0, 1)$
such that
\[
\spec (x) = \spec (y) = [0, 1], \quad x \precsim_B p, \quad
p \precsim_B (x - \ep)_{+} \oplus (y - \ep)_{+},
\quad {\mbox{and}}
\quad y \precsim_B (x - \ep)_{+}.
\]
Since $\cc{(x - \ep)_{+}}_B \ll \cc{x}_B$ and $0$ is a
limit point of $\spec ((x - \ep)_{+})$ as well as of
$\spec (x)$, by the previous case (see~(\ref{Eq_6318_Strong})),
we obtain $\gamma \in \Cu (B)$ such that
\begin{equation*}
k \gamma \leq \cc{x}_B \leq \cc{p}_B
\andeqn
\cc{(x - \ep)_{+}}_B \leq (k + 1) (2 m + 2) \gamma.
\end{equation*}
Now
\[
\cc{p}_B \leq 2 \cc{(x - \ep)_{+}}_B
\leq (k + 1) (4 m + 4) \gamma.
\]
Thus $B$ is $(4m + 3)$-almost divisible.
\end{proof}

\subsection{Pureness for Large Subalgebras}\label{Sec_6327_Pure}

We put the results of Section~\ref{Sec_6314_Comp}
and Section~\ref{Sec_6314_AD} together to get the the following result.

\begin{thm}\label{T_LargePure}
Let $A$ be an infinite dimensional simple stably finite \uca,
and let $B \S A$ be a subalgebra which is large in the sense
of Definition~4.1 of~\cite{Phl40}.
Then $A$ is pure if and
only if $B$ is pure.
\end{thm}

\begin{proof}
If $B$ is pure, then $B$ has $0$-comparison and is $0$-almost
divisible.
So $A$ has $0$-comparison by
Theorem~\ref{T_6314_CompThm} and is $1$-divisible by
Proposition~\ref{P_5Z28_Large_alm_dv}.
Together, these imply $A$ is pure by Theorem~\ref{T_PureEquiv}.

Conversely, suppose that $A$ is pure.
Then $A$ has
$0$-comparison and is $0$-almost divisible.
So $B$ has $0$-comparison by Theorem~\ref{T_6314_CompThm}
and is $3$-almost divisible by Proposition~\ref{P_5Z28_Large_alm_dv_C}.
Together, these imply $B$ is pure by Theorem~\ref{T_PureEquiv}.
\end{proof}

\section{Pureness of Crossed Products by
Automorphisms of $C (X, D)$}\label{Sec_PureCrossProd}

We establish pureness for crossed product C*-algebras
$C^* \bigl( \Z, \, C (X, D), \, \af \bigr)$ where $D$
is a simple unital pure C*-algebra, $X$ is a compact metric
space, and $\af$ induces a minimal homeomorphism $h \colon X \to X$.
Consequently, such crossed
products will have stable rank one, and under suitable
conditions will also have real rank zero.

\begin{ntn}\label{N_4Y12_C0UD}
For a locally \chs~$X$
and a \ca~$D$,
we identify $C_0 (X, D) = C_0 (X) \otimes D$
in the standard way.
For an open subset $U \subseteq X$,
we use the abbreviation
\[
C_0 (U, D)
= \big\{ a \in C_0 (X, D) \colon
{\mbox{$a (x) = 0$ for all $x \in X \setminus U$}} \big\}
\subseteq C_0 (X, D).
\]
This subalgebra is of course canonically isomorphic to
the usual algebra $C_0 (U, D)$ when $U$ is considered
as a locally \chs{} in its own right.
\end{ntn}

In particular,
if $Y \subseteq X$ is closed, then
\begin{equation}\label{Eq_4Y12_C0XYD}
C_0 (X \setminus Y, \, D)
= \big\{ a \in C_0 (X, D) \colon
{\mbox{$a (x) = 0$ for all $x \in Y$}} \big\}.
\end{equation}

In Lemma~\ref{L_1_GetAuto}, we recall from~\cite{ArBcPh2}
a general construction for automorphisms of $C_0 (X,D)$
which lie over homeomorphisms of $X$ in the sense
of Definition~1.2 of~\cite{ArBcPh2}. (This means that there
is a homeomorphism $h \colon X \to X$ such that the
corresponding action of $\Z$ on $C_0 (X)$ is generated by
$f \mapsto f \circ h^{-1}$, for $f \in C_0 (X)$).
This will be used repeatedly in the examples of
Section~\ref{Sec_Examples}.

\begin{lem}[Lemma 1.4 of \cite{ArBcPh2}]\label{L_1_GetAuto}
Let $X$ be a locally \chs, let $h \colon X \to X$ be a \hme,
and let $D$ be a \ca.
Then there is a one to one correspondence
between actions of $\Z$ on $C_0 (X, D)$ that lie over $h$ and
continuous functions from $X$ to $\Aut (D)$, given as follows.

For any function $x \mapsto \af_{x}$ from $X$ to $\Aut (D)$
such that $x \mapsto \af_{x} (d)$ is \ct{} for all $d \in D$,
there is an automorphism $\af \in \Aut ( C_0 (X, D))$
given by $\af (a) (x) = \af_x (a (h^{-1} (x)))$ for all
$a \in C_0 (X, D)$ and $x \in X$, and this automorphism lies
over~$h$.

Conversely, if $\af \in \Aut ( C_0 (X, D))$ lies over~$h$,
then there is a function $x \mapsto \af_{x}$ from $X$ to
$\Aut (D)$ such that $x \mapsto \af_{x} (d)$ is \ct{} for
all $d \in D$ and such that
$\af (a) (x) = \af_x (a (h^{-1} (x)))$ for all
$a \in C_0 (X, D)$ and $x \in X$.
\end{lem}

The following definition was introduced in Definition~2.3
of~\cite{ArBcPh2}, and is the analog of
Definition~7.3 of~\cite{Phl40}.

\begin{dfn}\label{D_1_OrbSubalg}
Let $X$ be a locally \chs,
let $h \colon X \to X$ be a \hme,
let $D$ be a \ca,
and let $\af \in \Aut ( C_0 (X, D))$
be an automorphism which lies over~$h$.
Let $u$ be
the standard unitary implementing the action of
$\af$ in the crossed product C*-algebra
$C^* \bigl( \Z, \, C_0 (X, D), \, \af \bigr)$.
Let $Y \subseteq X$ be a nonempty closed subset,
and, following~(\ref{Eq_4Y12_C0XYD}) in
Notation~\ref{N_4Y12_C0UD}, define the
corresponding {\emph{orbit breaking subalgebra}} of
$C^* \big( \Z,  C_0 (X, D), \af \big)$ to be
\[
C^* \big( \Z, \, C_0 (X, D), \, \af \big)_Y
= C^* \big(  C_0 (X, D), \, C_0 (X \setminus Y, \, D) u \big)
\subseteq C^* \big( \Z, \, C_0 (X, D), \, \af \big).
\]
\end{dfn}

\begin{lem}\label{L_2_OrbBreakPure}
Let $X$ be a compact metric space, let $h \colon X \to X$ be a
minimal homeomorphism, let $D$ be a simple unital pure C*-algebra,
and let $\af \in \Aut (C (X,D))$ lie over $h$.
Choose any nonempty
closed set $Y \subseteq X$ such that $\sint(Y) \neq \varnothing$.
Then
$C^* \bigl( \Z, \, C (X, D), \, \af \bigr)_{Y}$ is pure.
\end{lem}

\begin{proof}
By Theorem~3.17 of~\cite{ArBcPh2}, the algebra
$C^* \bigl( \Z, \, C (X, D), \, \af \bigr)_{Y}$ is a
recursive subhomogeneous algebra over $D$ in the sense of
Definition~3.2 of~\cite{ArBcPh2}, and so is pure by
Proposition~6.5 of~\cite{SethVil}.
\end{proof}

\begin{prp}\label{P_3_ApproxSubPure}
Let $X$ be a compact metric space, let $h \colon X \to X$ be a
minimal homeomorphism, let $D$ be a simple unital pure C*-algebra,
and let $\af \in \Aut (C (X,D))$ lie over $h$.
Choose any nonempty
closed set $Y \subseteq X$ such that $h^{n} (Y) \cap Y = \varnothing$
for all $n \in \Z \setminus \set{0}$.
Then
$C^* \bigl( \Z, \, C (X, D), \, \af \bigr)_{Y}$ is pure.
\end{prp}

\begin{proof}
For such a choice of $Y$, the algebra
$C^* \bigl( \Z, \, C (X, D), \, \af \bigr)_{Y}$ is a
direct limit (with injective maps) of recursive subhomogeneous
algebras over $D$ (of the form given in
Lemma~\ref{L_2_OrbBreakPure}) by
Proposition~4.2 of~\cite{ArBcPh2}.
Pureness follows by
Theorem~3.8 of~\cite{PerThiVil}.
\end{proof}

\begin{thm}\label{T_4_PureCrossProd}
Let $X$ be a compact metric space, let $h \colon X \to X$ be a
minimal homeomorphism, let $D$ be a simple unital pure C*-algebra,
and let $\af \in \Aut (C (X,D))$ lie over $h$.
Then the crossed
product $C^* \bigl( \Z, \, C (X, D), \, \af \bigr)$ is pure.
\end{thm}

\begin{proof}
Choose any point $x \in X$ and set $Y = \set{x}$.
Then
the minimality of $h$ implies that $h^{n}(Y) \cap Y =
\varnothing$ for $n \in \Z \setminus \set{0}$, and so
$C^* \bigl( \Z, \, C (X, D), \, \af \bigr)_{Y}$ is pure
by Proposition~\ref{P_3_ApproxSubPure}.

Since $D$ is simple and pure, it has strict comparison
of positive elements, and so Corollary~2.12(2)
of~\cite{ArBcPh2} implies that
$C^* \bigl( \Z, \, C (X, D), \, \af \bigr)_{Y}$
is a centrally large subalgebra of
$C^* \bigl( \Z, \, C (X, D), \, \af \bigr)$.
Pureness now follows from Theorem~\ref{T_LargePure}.
\end{proof}

\begin{cor}\label{C_5_CrossProdSR1}
Let $X$ be a compact metric space, let $h \colon X \to X$ be a
minimal homeomorphism, let $D$ be a simple unital pure C*-algebra,
and let $\af \in \Aut (C (X,D))$ lie over $h$.
Then the crossed
product $C^* \bigl( \Z, \, C (X, D), \, \af \bigr)$ has stable
rank one.
\end{cor}

\begin{proof}
This follows immediately from Theorem~\ref{T_4_PureCrossProd}
and Corollary~1.3 of~\cite{HLin}.
\end{proof}

The next result, which typifies when such crossed products have
real rank zero, requires the additional assumption of exactness
for~$D$.
This is not a major issue for our purposes, as the
examples in Section~\ref{Sec_Examples} mostly use choices for $D$
which are exact.
The exactness hypothesis is likely unnecessary,
provided that one uses quasitraces instead of traces.

\begin{cor}\label{C_6_CrossProdRR0}
Let $X$ be a compact metric space, let $h \colon X \to X$ be a
minimal homeomorphism, let $D$ be a simple unital pure exact
C*-algebra, and let $\af \in \Aut (C (X,D))$ lie over $h$.
Then the crossed product
$A = C^* \bigl( \Z, \, C (X, D), \, \af \bigr)$
has real rank zero if and only if the canonical map
$\rho \colon K_{0}(A) \to \Aff (T(A))$ has dense range.
\end{cor}

\begin{proof}
Theorem~\ref{T_4_PureCrossProd} implies that
$A$ is pure. It follows that $A$ has stable rank one by
Corollary~\ref{C_5_CrossProdSR1},
and that $\W (A)$ is almost unperforated
by Proposition~\ref{CP_6325_WA_Cmp}.
The exactness of $D$ implies that $A$ is also exact.
Now apply Proposition 7.1 of~\cite{Rord2004} to
obtain the given conclusion.
\end{proof}

In order to take advantage of
Corollary~\ref{C_6_CrossProdRR0}, we must know the
structure the tracial state space for
$C^* \bigl( \Z, \, C (X, D), \, \af \bigr)$.
This seems to be difficult in general, but is more
accessible when $D$ has a unique tracial state.
We do not know what happens when $D$ has a unique quasitrace
and this quasitrace is not a trace.

The following result is surely well known, but we have
not been able to find a reference.

\begin{lem}\label{L_1_TracesTensorProd}
Let $A$ and $B$ be simple unital C*-algebras.
Assume $B$ has a unique tracial state $\sm$.
Then the map $\rho \mapsto \rho \otimes \sm$
is a bijection from $T(A)$ to $T(A \otimes_{\min} B)$.
\end{lem}

\begin{proof}
Let $\tau$ be a tracial state on $A \otimes_{\min} B$.
For fixed $a \in A_{+}$, define a map $\om_{a}$ on
$B$ by $\om_{a} (b) = \tau (a \otimes b)$.
Then
$\om_{a}$ is a bounded linear functional on $B$,
and is positive.
For any $b,c \in B$, we have
\begin{align*}
\om_{a} (bc) - \om_{a}(cb) &= \tau (a \otimes bc) -
\tau (a \otimes cb) \\
&= \tau (a \otimes bc - a \otimes cb) \\
&= \tau ( (a \otimes b) (1_{A} \otimes c) -
(1_{A} \otimes c) (a \otimes b) ) = 0
\end{align*}
and so $\om_{a}$ has the trace property.
Therefore,
it is a nonnegative multiple of $\sm$, say
$\om_{a} (b) = \rho (a) \sm (b)$.
Writing
$a \in A$ as a linear combination of positive
elements allows the extension of $\rho$ to a map
$A \to {\mathbb{C}}$.
Positivity and the condition
that $\rho (1) = 1$ are immediate.
The trace
property of $\rho$ is easily verified via the
calculation
\[
\rho (xy) = \rho(xy) \sm (1_{B}) = \om_{xy}(1_{B})
= \tau (xy \otimes 1_{B}) = \tau (yx \otimes 1_{B})
= \rho (yx) \sm(1_{B}) = \rho (yx).
\]
Additivity of $\rho$ on $A_{+}$ follows from the
observation (using the linearity of $\tau$) that
\begin{align*}
\rho (a_{1} + a_{2}) \sm (b)
= \om_{a_{1} + a_{2}} (b)
&= \tau ((a_{1} + a_{2}) \otimes b) \\
&= \tau (a_{1} \otimes b) + \tau (a_{2} \otimes b)
= \rho (a_{1}) \sm (b) + \rho (a_{2}) \sm (b)
\end{align*}
for any $a_{1}, a_{2} \in A_{+}$, and linearity
of $\rho$ follows.
Hence $\rho$ is a tracial state
on A.
That $\rho \otimes \sm = \tau$ follows from
density of the linear span of the elementary tensors.
This gives surjectivity of the map
$\rho \mapsto \rho \otimes \sm$,
and injectivity follows easily by evaluation on
$a \otimes 1_{B}$.
\end{proof}

\begin{lem}\label{L_2_AveApprox}
Let $D$ be a simple unital C*-algebra, let $X$ be
a compact metric space, let $h \colon X \to X$ be
a minimal homeomorphism, and obtain an action
$\af \colon X \to \Aut (D)$ as in
Lemma~\ref{L_1_GetAuto}.
For every $a \in
C^* \bigl( \Z, \, C (X, D), \, \af \bigr)$ and every
$\ep > 0$, there exist an $n \in \N$ and
$s_1 , \ldots , s_n \in C (X,D)$ such that
$\lVert s_k (x) \rVert = 1$ for all $x \in X$, and
\[
\Big\lVert E(a) - \frac{1}{n}
\sum_{k=1}^{n} s_k a s_k^* \Big\rVert < \ep.
\]
\end{lem}

\begin{proof}
Choose $N \in \Nz$ and $f_k \in C (X,D)$ for
$-N \leq k \leq N$ such that
$b = \sum_{k = -N}^{N} f_k u^k$ satisfies
$\lVert a - b \rVert < \ep/2$.
Set
$F = \set{-N, \ldots,N}$.
By Proposition 11.1.19
of~\cite{GKPT}, there exist $n \in \N$ and
$t_1 , \ldots , t_n \in C (X)$ such that,,
for $1 \leq k \leq n$,
$\vert t_k (x) \vert = 1$ for all $x \in X$
and such that, for all
$x \in X$ and $j \in F \setminus {0}$, we have
\[
\frac{1}{n} \sum_{k=1}^{n} t_k (x)
\overline{ t_k (h^{-j}(x)) } = 0.
\]
Set $s_k = t_k \otimes 1_D$.
Define
$P \colon C^* \bigl( \Z, \, C (X, D), \, \af \bigr)
\to C^* \bigl( \Z, \, C (X, D), \, \af \bigr)$ by
\[
P(c) = \frac{1}{n} \sum_{k=1}^{n} s_k c s_k^*.
\]
We first compute
\begin{align*}
P(f_0) = \frac{1}{n} \sum_{k=1}^{n} s_k f_0 s_k^*
&= \frac{1}{n} \sum_{k=1}^{n} (t_k \otimes 1_D)
f_0 (t_k \otimes 1_D)^* \\ &= f_0 \cdot \frac{1}{n}
\sum_{k=1}^{n} (t_k \otimes 1_D) (t_k \otimes 1_D)^*
= f_0 = E(f_0).
\end{align*}
It is easy to show that, for $j \neq 0$ and any
$x \in X$, we have
\[
\af^j (s_k^*) (x) = \af^j (t_k^* \otimes 1_D) (x)
= \overline{ t_k (h^{-j}(x)) } \cdot 1_D,
\]
and hence, for any $x \in X$,
\[
\frac{1}{n} \sum_{k=1}^{n} s_k (x) \af^j (s_k^*) (x)
= \Big[ \frac{1}{n} \sum_{k=1}^{n} t_k (x)
\overline{ t_k (h^{-j}(x)) } \Big] \cdot 1_D = 0.
\]
It follows that, for any $j \in F \setminus \set{0}$,
\[
P(f_j u^j) =
\frac{1}{n} \sum_{k=1}^{n} s_k f_j u^j s_k^*
= \frac{1}{n} \sum_{k=1}^{n} f_j s_k u^j s_k^*
u^{-j} u^j = f_j \Big[
\frac{1}{n} \sum_{k=1}^{n} s_k \af^j (s_k^*) \Big] u^j
= 0.
\]
This implies that $P(b) = E(b)$, and observing that
$\lVert s_k \rVert = 1$ for $1 \leq k \leq n$, we
obtain $\lVert P \rVert \leq 1$.
Now,
\begin{align*}
\lVert E(a) - P(a) \rVert &= \lVert E(a) - E(b) \rVert
+ \lVert E(b) - P(b) \rVert + \lVert P(b) - P(a) \rVert \\
&< \ep/2 + 0 + \ep/2 \\
&= \ep.
\end{align*}
This is the desired result.
\end{proof}

\begin{prp}\label{P_1_TracialStates}
Let $D$ be a simple unital exact C*-algebra with
a unique tracial state $\sm$, let $X$ be a compact
metric space, let $h \colon X \to X$ be a minimal
homeomorphism, and obtain an action
$\af \colon X \to \Aut (D)$ as in
Lemma~\ref{L_1_GetAuto}.
For an $h$-invariant Borel probability
measure $\mu$ on $X$, define a functional $\tau_{\mu}$
on $C^* \bigl( \Z, \, C (X, D), \, \af \bigr)$ by
\[
\tau_{\mu}(a) = \int_{X} \sm (a (x)) \, d \mu (x).
\]
Then $\tau_{\mu}$ is always a tracial state, and
$\mu \mapsto \tau_{\mu}$ is an affine bijective
homeomorphism from the space $M_{h}(X)$ of $h$-invariant
Borel probability measures on $X$ to the tracial states
on $C^* \bigl( \Z, \, C (X, D), \, \af \bigr)$.
\end{prp}

\begin{proof}
For an invariant probability measure $\mu$, define
a tracial state $\rho_{\mu}$ on $C (X, D)$ by the formula
\[
\rho_{\mu} (f) = \int_X \sm (f (x)) d \mu (x).
\]
Obviously $\rho_{\mu}$ is a tracial state,
and it is $\af$-invariant because $\sm$ is invariant
under all automorphisms of $D$.
Lemma~\ref{L_1_TracesTensorProd} implies that these
are all the $\af$-invariant tracial states on $C (X, D)$.

Let
$E \colon C^* \bigl( \Z, \, C (X, D), \, \af \bigr) \to C (X, D)$
be the standard conditional expectation.
Then for any
$\af$-invariant tracial state $\rho_{\mu}$ defined above,
$\tau_{\mu} = \rho_{\mu} \circ E$ defines a tracial state
on $C^* \bigl( \Z, \, C (X, D), \, \af \bigr)$.
Now let
$\tau$ be any tracial state on
$C^* \bigl( \Z, \, C (X, D), \, \af \bigr)$.
Restricting
$\tau$ to $C (X, D)$ gives an $\af$-invariant tracial
state on $C (X, D)$, and so there is an invariant
measure $\mu$ such that $\tau \vert_{C (X,D)} = \rho_{\mu}$.
We want to show that $\tau = \tau_{\mu}$.
Since $\tau (f)
= \tau \vert_{C (X,D)} (f) = \rho_{\mu} (f)$
for every $f \in C (X,D)$, it suffices
to prove that, for every $\ep> 0$ and every $a \in
C^* \bigl( \Z, \, C (X, D), \, \af \bigr)$, we have
$\vert \tau (a) - \tau (E(a)) \vert < \ep$.
Apply
Lemma~\ref{L_2_AveApprox} to obtain $n \in \N$ and
$s_1 , \ldots , s_n \in C (X,D)$ such that
$\lVert s_k \rVert = 1$ for all $k$ and
\[
\Big\lVert E(a) - \frac{1}{n} \sum_{k=1}^{n} s_k a s_k^*
\Big\rVert < \ep.
\]
For each $k$, the trace property gives $\tau (a) =
\tau (a s_k^* s_k ) = \tau (s_k a s_k^* )$, and hence
$\ds{ \tau (a) = \tau \Big( \tfrac{1}{n}
\sum_{k=1}^{n} s_k a s_k^* \Big) }$. It follows that
\[
\vert \tau (a) - \tau (E(a)) \vert = \Big\vert \tau
\Big( \tfrac{1}{n} \sum_{k=1}^{n} s_k a s_k^* \Big) -
\tau (E(a)) \Big\vert \leq \Big\lVert
\frac{1}{n} \sum_{k=1}^{n} s_k a s_k^* - E(a) \Big\rVert
< \ep,
\]
as required.
\end{proof}

\begin{cor}\label{C_1_RR0_Condition}
Let $D$ be a simple unital exact C*-algebra with
a unique tracial state, let $X$ be a compact
metric space, let $h \colon X \to X$ be a minimal
homeomorphism, and obtain an action
$\af \colon X \to \Aut ( C (X, D))$ as in
Lemma~\ref{L_1_GetAuto}.
If the range of the
canonical map
$K_{0} ( C^* ( \Z, X, h ) ) \to
\Aff (T( C^* ( \Z, X, h ) )$
is dense, then the range of the canonical map
$K_{0} ( C^* \bigl( \Z, \, C (X, D), \, \af \bigr) ) \to
\Aff (T( C^* \bigl( \Z, \, C (X, D), \, \af \bigr) ))$
is dense.
\end{cor}

\begin{proof}
Since $D$ has a unique tracial state,
Proposition~\ref{P_1_TracialStates} implies that
\[
T( C^* \bigl( \Z, \, C (X, D), \, \af \bigr) ) \cong
T ( C^* ( \Z, X, h ) ) \cong M_{h} (X).
\]
Since
$C^* ( \Z, X, h )$ is a subalgebra of
$C^* \bigl( \Z, \, C (X, D), \, \af \bigr)$, the result
follows.
\end{proof}

\begin{cor}\label{C_6327_RRZ}
Adopt the hypotheses of Corollary~\ref{C_1_RR0_Condition}.
If the range of the canonical map
$K_{0} ( C^* ( \Z, X, h ) ) \to \Aff (T ( C^* ( \Z, X, h ) )$
is dense, then
$C^* \bigl( \Z, \, C (X, D), \, \af \bigr)$ has real rank zero.
\end{cor}

\begin{proof}
Combine Corollary~\ref{C_1_RR0_Condition}
with Corollary~\ref{C_6_CrossProdRR0}.
\end{proof}

\section{Pureness of Crossed products by Compact Group Actions with
 finite Rokhlin Dimension with commuting towers}\label{Sec_6326_FinRD}

In this section, we apply approximation results for crossed products,
where the action has finite Rokhlin dimension with commuting towers,
in order to establish pureness for certain crossed products.
The main result we use, due to Gardella, Hirshberg, and Santiago,
states that if $\alpha \colon G \to \Aut(A)$
has finite Rokhlin dimension with commuting towers,
then the crossed product $C^* \bigl( G, \, A, \, \af \bigr)$
can be locally approximated by a continuous $C (X)$-algebra
whose fibers are isomorphic to $A \otimes K(L^2 (G))$.

We begin with the following notation and definitions.

\begin{ntn}\label{N_6326_TrA}
Let $G$ be a compact group.
We write $\mathtt{Lt} \colon G \to \Aut (C (G))$
for the left action of $G$ on $C (G)$,
given by $\mathtt{Lt}_g (f) (t) = f (g^{-1} t)$.
\end{ntn}

The following lemma for finite Rokhlin dimension with commuting towers,
due to Gardella,
characterizes it in terms of elements of the C*-algebra itself.
See Definition~3.2 and Lemma~3.7 of~\cite{gardella17}.

\begin{lem}\label{D_6326_SeqSpl}
Let $G$ be a second countable compact group. Let $A$ be a separable C*-algebra,
and let $\alpha \colon G \to \Aut (A)$ be an action.
Let $d \in \Z \setminus \set{0}$.
Then $\RD (\alpha) \leq d$ if and only if, for any finite sets
$F \subseteq A$ and $S \subseteq C (G)$ and any $\ep > 0$,
there exist $d + 1$ completely positive contractive maps
\[
\psi_0, \psi_1, \ldots, \psi_d \colon C (G) \to A
\]
satisfying the following conditions:
\begin{enumerate}
\item
For any $a \in F$ and $f\in S$,
$\|[\psi_j(f), a]\| < \ep$ for $j = 0, 1, \ldots, d$.
\item
For any $f\in S$ and $g\in G$,
$\|\alpha_g (\psi_j(f)) - \psi_j(\mathtt{Lt}_g (f)) \| < \ep$
for $j = 0, 1, \ldots, d$
\item
$\|\psi_j(f_1) \psi_j(f_2) \| < \ep$
whenever $f_1$ and $f_2$ in $S$ are orthogonal.
\item
For any $a \in F$, $\|\sum_{j=0}^d \psi_j(1_{C (G)})a - a\| < \ep$.
\item
For any $f_1,f_2 \in S$,
$\|[\psi_j(f_1), \psi_k(f_2)]\| < \ep$ for $j, k = 0, 1, \ldots, d$.
\end{enumerate}
\end{lem}

The following result is an important structure theorem
due to Gardella, Hirshberg, and Santiago.

\begin{prp}[Corollary 4.6 of~\cite{GdlHrbStg}]\label{2003_structure_cx}
Let $G$ be a second countable compact group,
let $X$ be a compact Hausdorff space,
and let $A$ be a separable C*-algebra.
Let $G \curvearrowright X$ be a free continuous action of $G$ on~$X$,
and let $\alpha \colon G \to \Aut (A)$ be an action of $G$ on~$A$.
Equip the C*-algebra $C (X, A)$ with the diagonal action of $G$,
denoted by $\gamma$.
Then the crossed product C*-algebra
$C^* \bigl( G, \, C(X,A), \, \gm \bigr)$
is a continuous $C (X/G)$-algebra,
each of whose fibers is isomorphic to $A \otimes K(L^2 (G))$.
If $G$ is furthermore a Lie group,
then $C^* \bigl( G, \, C(X,A), \, \gm \bigr)$ is the
section algebra of a locally trivial bundle over $X/G$
with fibers isomorphic to $A \otimes K(L^2 (G))$.
\end{prp}

We now recall the definition of sequentially split homomorphisms
(Definition~2.1 of~\cite{barlak}).

\begin{dfn}
Let $A$ and $B$ be two C*-algebras. A homomorphism
$\ph\colon A \rightarrow B$ is \emph{sequentially split},
if there exists a commutative diagram of homomorphisms
of the form
\[
\begin{tikzcd}
A \arrow[rr, "\iota"] \arrow[dr, "\ph"'] & & A_\infty\\
& B \arrow[ur, "\ps" ' ] &
\end{tikzcd}
\]
where the horizontal map $\iota$ is the canonical inclusion.
\end{dfn}

The following theorem, which is Proposition~4.11 of~\cite{GdlHrbStg},
is a key approximation result used to establish permanence properties
for crossed products by actions with finite Rokhlin dimension with
commuting towers. Its statement involves the notion of a
\emph{positively existential embedding}, which, in the setting of
separable C*-algebras, coincides with that of a sequentially
split homomorphism.

\begin{thm}\label{2003_rokhlin_dimension}
Let $\alpha \colon G \to \Aut (A)$ be an action
of a second countable compact group on a separable C*-algebra
such that $\RD (\alpha) < \infty$.
Then there exist a compact metric space $X$
and a free action $G \curvearrowright X$ such that, by endowing $C (X, A)$ with the diagonal $G$-action $\gamma$, the
canonical embedding
\[
\rho \colon C^* \bigl( G, \, A, \, \af \bigr) \to
C^* \bigl( G, \, C(X,A), \, \gm \bigr)
\]
is sequentially split.
Furthermore, if $G$ is finite dimensional,
then $X$ may be chosen to be finite dimensional.
\end{thm}

In \cite{barlak}, it is shown that the existence of a sequentially split homomorphism from the C*-algebra $A$ to the C*-algebra $B$
implies that several regularity properties pass from $B$ to $A$.
The following results show that this is also the case for pureness.

\begin{prp}\label{0704_seq_split_comparison_divisibility}
Let $A$  and $B$ be separable C*-algebras, and
$\ph \colon A\rightarrow B$ be a sequentially split homomorphism. If $B$ is pure, then so is $A$.
\end{prp}

\begin{proof}
First observe that if $\varphi \colon A \to B$ is a sequentially split homomorphism, then the induced (amplified) homomorphism
\[
\tilde{\varphi} \colon A \otimes K \to B \otimes K
\]
is again sequentially split. By a mild abuse of notation, we continue to denote this amplified map by $\varphi$. We thus obtain a commutative diagram
\[ \begin{tikzcd} A \otimes K \arrow[rr, "\iota"] \arrow[dr, "\ph"'] & & (A \otimes K )_\infty\\ & B \otimes K \arrow[ur, "\ps" ' ] & \end{tikzcd} \]
where $\iota$ denotes the canonical inclusion of $A \otimes K$ into $(A \otimes K)_\infty$ as constant sequences.

Since $B$ is pure, $\Cu(B)$ has $0$-almost comparison and is $0$-almost divisible. Moreover, $\Cu(A)$ has $0$-almost comparison by Theorem~2.9 of~\cite{barlak}. It remains to show that $\Cu(A)$ is $0$-almost divisible. For this, let
$\ep > 0$, $k\in \Nz$, and $a\in (A \otimes K)_+$. Fix
$\left( a_n\right)_{n=1}^\infty$ to be a representative
sequence of $(\ps \circ \ph) (a)$. Since
$(\ps \circ \ph) (a) = \iota (a)$, we can choose $N_1$,
such that for all $n \geq N_1$,
\begin{equation*}
\| a_n - a\| < \frac{\ep}{4}.
\end{equation*}
Thus, for all $n \geq N_1$, we get the following subequivalences
\begin{equation}\label{0104_divisibility_seq}
\left( a_n - \frac{\ep}{4}\right)_+ \precsim_A a\quad\text{and}\quad (a- \ep)_+ \precsim_A \left(a_n -\frac{\ep}{2}\right)_+ \precsim_A \left(a_n - \frac{7\ep}{16} \right)_+.
\end{equation}
Furthermore, since $\Cu(B)$ is $0$-almost divisible, there exists $b \in (B \otimes K)_+$ such that
\begin{equation}\label{0104_divisibility_seq_split}
k \cc{b}_B \leq \inn{\left(\ph(a)-\frac{\ep}{4}\right)_+}_B\quad\text{and}\quad \inn{\left(\ph(a)-\frac{3\ep}{8}\right)_+}_B \leq (k+1) \cc{b}_B.
\end{equation}
From the second inequality in the above equation, we obtain
\begin{equation*}
\left(\ph(a)-\frac{3\ep}{8}\right)_+\precsim_B b \otimes 1_{k+1},
\end{equation*}
and hence
\begin{equation*}
\left((\ps \circ \ph)(a)-\frac{3\ep}{8}\right)_+\precsim_{{(A \otimes K)}_\infty} \ps(b) \otimes 1_{k+1}.
\end{equation*}
Let $e \in (A \otimes K)_\infty$ be such that
\begin{equation*}
\left\|e^* \left(\ps (b) \otimes 1_{k+1}\right) e - \left((\ps \circ \ph)(a)- \frac{3\ep}{8}\right)_+\right\| < \frac{ \ep}{32}.
\end{equation*}
Fix representative sequences $(e_n)_{n=1}^\infty$ and $(b_n)_{n=1}^\infty$ for $e$ and $\psi(b)$, respectively, and choose $N_2 \in \Nz$ such that, for all $n \geq N_2$,
\begin{equation*}
\left\| e_n^* (b_n \otimes 1_{k+1} ) e_n - \left( a_n - \frac{3\ep}{8}\right)_+\right\| < \frac{\ep}{32}.
\end{equation*}
It follows that, for all $n \geq N_2$,
\begin{equation}\label{0104_split_1}
\left( \left(a_n -\frac{3\ep}{8}\right)_+ - \frac{\ep}{16} \right)_+ \precsim_A \left(e_n^* (b_n \otimes 1_{k+1}) e_n -\frac{\ep}{32}\right)_+.
\end{equation}
Let $M\in \Nz$ be such that $\|e_n\|^2 \leq M$ for all $n$. By
Lemma~1.4(8) of~\cite{Phl40},
\begin{equation}\label{0104_split_7}
\left(\left(a_n -\frac{3\ep}{8}\right)_+ - \frac{\ep}{16}\right)_+=  \left(a_n - \frac{7\ep}{16}\right)_+,
\end{equation}
and, by Lemma~1.4(6) of~\cite{Phl40} at the second step and
Lemma~1.7 of~\cite{Phl40} at the third step, we obtain
\begin{equation}\label{0104_split_2}
\begin{split}
\left(e_n^*(b_n \otimes 1_{k+1}) e_n -\frac{\ep}{32}\right)_+ & =
\left(e_n^* (b_n \otimes 1_{k+1})^{1/2} (b_n \otimes 1_{k+1})^{1/2} e_n -\frac{\ep}{32}\right)_+ \\
& \sim_A \left((b_n \otimes 1_{k+1})^{1/2} e_n e_n^* (b_n \otimes 1_{k+1})^{1/2} - \frac{\ep}{32}\right)_+\\
&\precsim_A \left(\|e_n\|^2 (b_n \otimes 1_{k+1}) -\frac{\ep}{32}\right)_+\\
& \precsim_A \left(b_n \otimes 1_{k+1}- \frac{\ep}{32 M}\right)_+.
\end{split}
\end{equation}
From \eqref{0104_split_1}, \eqref{0104_split_7}, and \eqref{0104_split_2}, we see that
\begin{equation}\label{0104_split_div_second}
\left(a_n - \frac{7\ep}{16}\right)_+ \precsim_A \left(b_n \otimes 1_{k+1} - \frac{\ep}{32 M}\right)_+.
\end{equation}
Moreover, the first inequality in \eqref{0104_divisibility_seq_split} implies that
\begin{equation*}
\ps(b) \otimes 1_{k} \precsim_{(A\otimes K)_\infty} \left( (\ps \circ \ph) (a) -\frac{ \ep}{4}\right)_+.
\end{equation*}
Choose $c\in (A\otimes K)_\infty$ such that
\begin{equation*}
\left\| c^* \left((\ps\circ \ph)(a)-\frac{\ep}{4}\right)_+ c - \ps(b) \otimes 1_k\right\| < \frac{\ep}{32 M}.
\end{equation*}
Let $(c_n)_{n=1}^\infty$ be a representative sequence for $c$, and let $N_3\in \Nz$ be such that, for all $n\geq N_3$,
\begin{equation}
\left\|c_n^* \left(a_n - \frac{\ep}{4}\right)_+ c_n - b_n \otimes 1_{k}\right\| <\frac{\ep}{32M}.
\end{equation}
Then, for all $n\geq N_3$,
\begin{equation}\label{0104_split_div_first}
\left(b_n \otimes 1_k -\frac{\ep}{32 M}\right)_+ \precsim_A \left( a_n -\frac{\ep}{4}\right)_+.
\end{equation}
Finally, let $l:= \max \{N_1, N_2, N_3\}$. Using the first part of \eqref{0104_divisibility_seq} together with \eqref{0104_split_div_first}, we obtain
\[
k\inn{\left(b_l - \frac{\ep}{32 M}\right)_+}_A \leq \inn{\left( a_l- \frac{\ep}{4} \right)_+}_A \leq  \inn{a}_A,
\]
and by second part of (\ref{0104_divisibility_seq_split}) together with (\ref{0104_split_div_second}), we obtain
\[
\inn{(a- \ep)_+ }_A \leq \inn{\left(a_l -\frac{7\ep}{16}\right)_+}_A \leq (k+1) \inn{\left(b_l -\frac{\ep}{32 M}\right)_+}_A,
\]
which gives $0$-almost divisibility of $\Cu(A)$.
\end{proof}

In Theorem~4.24 of~\cite{GdlHrbStg},
it is shown that if $G$ is a compact group
and $\alpha \colon G \to \Aut(A)$ is an action with $\RD(\alpha) \leq 1$,
and if $A$ has stable rank one and $K_1(I) = 0$
for all ideals $I \subseteq A$,
then strict comparison passes from $A$ to the crossed product
$C^* ( G, \, A, \, \af )$.
The proof relies on Theorem~2.6 of~\cite{AnBoPe}, which establishes that, under these hypotheses, pointwise Cuntz subequivalence in the $C(X/G)$-algebra implies global Cuntz subequivalence. In light of recent developments in \cite{SethVil}, we obtain strict comparison and almost divisibility for the crossed product under the weaker assumption of finite Rokhlin dimension with commuting towers.
We state the following additional lemma before turning to the main result.

\begin{lem}\label{L_6327_CtFld}
Let $D$ be a simple pure \ca, let $X$ be a \chs,
and let $A$ be the section algebra of a locally trivial \ct{} field
over~$X$ with fiber~$D$.
Then $A$ is pure.
\end{lem}

When $D$ is unital, this is
Corollary~6.6 of \cite{SethVil}.
In our application, $D$ is not unital.

\begin{proof}[Proof of Lemma~\ref{L_6327_CtFld}]
Since $A$ is the section algebra of a locally trivial continuous field over $X$ with fiber $D$, repeated application of Lemma~2.4 of~\cite{Dad2009} shows that $A$ admits the structure of a recursive subhomogeneous algebra over $D$. The usual definition of recursive subhomogeneous algebras over $D$ (Definition~3.2 of~\cite{ArBcPh2}) assumes that $D$ is unital; however, it extends in a straightforward way to the nonunital case, and the resulting algebra still admits an iterated pullback structure of the same form. Moreover, the same argument as in Proposition~6.5 of~\cite{SethVil} shows that $A$ is pure.
\end{proof}

\begin{thm}\label{T_6325_RDim}
Let $\alpha \colon G \to \Aut (A)$ be an action of a compact Lie group
on a simple separable C*-algebra $A$ such that
$\RD (\alpha) < \infty$.
If $A$ is pure, then so is $C^* \bigl( G, \, A, \, \af \bigr)$.
\end{thm}

\begin{proof}
Let $X$ be the compact metric space
obtained from Theorem~\ref{2003_rokhlin_dimension}.
By Proposition~\ref{2003_structure_cx},
$C^* \bigl( G, \, C(X,A), \, \gm \bigr)$ is a locally trivial
continuous $C (X/G)$-algebra,
each of whose fibers is isomorphic to $A \otimes K(L^2 (G))$,
and hence is pure by Lemma~\ref{L_6327_CtFld}.
$C^* \bigl( G, \, A, \, \af \bigr)$ is thus  pure by
Proposition~\ref{0704_seq_split_comparison_divisibility}.
\end{proof}

\section{Pureness of Crossed Products by Compact Group
 Actions with the Tracial Rokhlin Property}\label{Sec_6326_RTRP}

In this section, we prove that pureness passes to crossed products
by actions of compact groups which have
the restricted tracial Rokhlin property with comparison,
in the sense of Definition~2.1 of~\cite{MohPhil2021TRP}.
We give independent results for comparison and almost divisibility.
We use $\W (A)$ in the proofs since the theorems we cite
are stated for $\W (A)$. The third and fourth author in \cite{MohPhil2021TRP} defined the (restricted) tracial Rokhlin property with comparison for compact group actions on unital simple C*-algebras. Examples that satisfy this definition are in \cite{MohPhil2021TRP2}. We begin with the following definition.

\begin{dfn}\label{traR}
Let $A$ be an infinite dimensional simple unital \ca,
and let $\alpha \colon G \to \Aut (A)$ be
an action of a second countable compact group $G$ on~$A$.
The action $\alpha$ has the
\emph{restricted tracial Rokhlin property with comparison}
if for every finite set $F \subseteq A$,
every finite set $S \subseteq C (G)$,
every $\varepsilon > 0$, every $x \in A_{+} \setminus \{ 0 \}$,
and every $y \in (A^{\alpha})_{+} \setminus \{ 0 \}$,
there exist a projection $p \in A^{\alpha}$ and a
unital completely positive map
$\varphi \colon (C (G), \mathtt{Lt}) \to (p A p, \alpha)$
such that the following hold.
\begin{enumerate}
\item\label{Item_893_FS_equi_cen_multi_approx}
$\varphi$ is an $(S, F, \varepsilon)$-approximately equivariant
central multiplicative map (Definition~1.5 of~\cite{MohPhil2021TRP}).
\item\label{1_pxcompactsets}
$1 - p \precsim_{A} x$.
\item\label{1_pycompactsets}
$1 - p \precsim_{A^{\alpha}} y$.
\item\label{1_ppcompactsets}
$1 - p \precsim_{A^{\alpha}} p$.
\setcounter{TmpEnumi}{\value{enumi}}
\end{enumerate}
\end{dfn}

\begin{thm}\label{ComparisonTRP}
Let $A$ be a stably finite simple separable infinite dimensional
unital C*-algebra,
let $G$ be a second countable compact group, and let
$\alpha \colon G \to \Aut (A)$ be an action
which has the restricted tracial Rokhlin property with comparison.
Let $m \in \mathbb{N}$.
Suppose $A$ has $m$-comparison.
Then $A^{\alpha}$ has $m$-comparison.
\end{thm}

\begin{proof}
The proof is similar to the proof that $m$-comparison passes from
a simple \ca{} to a large subalgebra, part of Theorem~\ref{T_6314_CompThm}.
By Proposition~\ref{CP_6325_WA_Cmp}, it suffices to prove
$m$-comparison in $\W (A^{\alpha})$.

Let $x, a_0, a_1, \ldots, a_m \in (M_{\infty} (A^{\alpha}))_{+}$ satisfy
$\cc{x}_{A^{\alpha}} \les \cc{a_j}_{A^{\alpha}}$ for $j = 0, 1, \ldots, m$.
\Wolog{} $x \neq 0$.
By the definition of $\les$ (Definition~\ref{LeqS}),
for $j = 0, 1, \ldots, m$ there is $k_j \in \N$
such that $(k_j + 1) \ccAf{x} \leq k_j \ccAf{a_j}$.
In particular, $a_j \neq 0$ for $j = 0, 1, \ldots, m$.

First assume $\cc{x}_{A^{\alpha}}$ is not the class of a projection.
Let $\varepsilon > 0$ be arbitrary.
We will find $z \in (M_{\infty} (A^{\alpha}))_{+}$
and, for $j = 0, 1, \ldots, m$,
we will find $b_j \in (M_{\infty} (A^{\alpha}))_{+}$
such that $0$ is a limit point of $\spec (z)$,
$(x - \ep)_{+} \precsim_{A^{\alpha}} z$,
and, for $j = 0, 1, \ldots, m$, $\ccA{z} \les \ccA{b_j}$,
$0$ is a limit point of $\spec (b_j)$, and $\ccAf{b_j} \leq \ccAf{a_j}$.

Choose $\rh \in (0, \ep)$ such that $\spec (x) \cap (\rh, \ep) \neq \E$.
Choose $\ld \in \spec (x) \cap (\rh, \ep)$.
Choose a \cfn{} $f \colon [0, \I) \to [0, 1]$ such that
$\supp (f) \S (\rh, \ep)$ and $f (\ld) \neq 0$.
The algebra $A^{\alpha}$ is simple
by Theorem~3.1 of~\cite{MohPhil2021TRP} and not of Type~I
by Proposition~3.2 of~\cite{MohPhil2021TRP}.
Therefore Lemma~\ref{L_3618_Sp01} provides
$c \in \bigl( {\overline{f (x) (M_{\infty} (A^{\alpha})) f (x)}} \bigr)_{+}$
such that $\spec (c) = [0, 1]$.
Set $z = (x - \ep)_{+} + c$.
Then $0$ is a limit point of $\spec (z)$ and
$(x - \ep)_{+} \precsim_{A^{\alpha}} z \precsim_{A^{\alpha}} x$.

Define
\[
I = \bigl\{ j \in \{0, 1, \ldots, m\} \colon
 {\mbox{$\cc{a_j}_{A^{\alpha}}$
    is the class of a nonzero projection}} \bigr\}.
\]

For $j \not\in I$, set $b_j = a_j$.
Then $0$ is a limit point of $\spec (b_j)$.
Also,
\[
\ccAf{z} \leq \ccAf{x}
\andeqn
(k_j + 1) \ccAf{x} \leq k_j \ccAf{b_j}.
\]
The second part says that $\cc{z}_{A^{\af}} \les \cc{b_j}_{A^{\af}}$.

Now suppose $j \in I$.
We may as well assume that $a_j$ is a \pj.
Use Lemma~\ref{L_5Z26_InterpPj} to choose
$b_j, s_j \in (a_j (M_{\infty} (A^{\alpha})) a_j)_{+}$ and $\dt \in (0, 1)$
such that
\[
\spec (b_j) = \spec (s_j) = [0, 1],
\qquad
b_j \precsim_{A^{\alpha}} a_j,
\]
\[
a_j \precsim_{A^{\alpha}} (b_j - \dt)_{+} \oplus (s_j - \dt)_{+},
\qquad {\mbox{and}} \qquad
2 k_j \ccAf{s_j} \leq \ccAf{(b_j - \dt)_{+}}.
\]
Then
\[
\begin{split}
(2 k_j + 2) \ccAf{z}
& \leq (2 k_j + 2) \ccAf{x}
\\
& \leq 2 k_j \ccAf{a_j}
  \leq 2 k_j \ccAf{b_j} + 2 k_j \ccAf{s_j}
  \leq (2 k_j + 1) \ccAf{b_j}.
\end{split}
\]
Thus, here too $\ccAf{z} \les \ccAf{b_j}$,
$0$ is a limit point of $\spec (b_j)$,
and $\ccA{b_j} \leq \ccA{a_j}$ because $\ccAf{b_j} \leq \ccAf{a_j}$.

Passing to~$A$, we have $\ccA{z} \les \ccA{b_j}$ for $j = 0, 1, \ldots, m$.
Since $A$ has $m$-comparison, it follows that
$\ccA{z} \leq \sum_{j = 0}^m \ccA{b_j}$.
Since $0$ is a limit point of $\spec (z)$
and of $\spec (b_j)$ for $j = 0, 1, \ldots, m$,
Proposition~4.10 of~\cite{MohPhil2021TRP} implies that
$\ccAf{z} \leq \sum_{j = 0}^m \ccAf{b_j}$.
Therefore
\[
\ccAf{(x - \ep)_{+}}
 \leq \ccAf{z}
 \leq \sum_{j = 0}^m \ccAf{b_j}
 \leq \sum_{j = 0}^m \ccAf{a_j}.
\]

We have proved that for all $\ep > 0$ we have
\[
\ccAf{(x - \ep)_{+}} \leq \sum_{j = 0}^m \ccAf{a_j}.
\]
So $\ccAf{x} \leq \sum_{j = 0}^m \ccAf{a_j}$.
This completes the case that $\ccAf{x}$ is not the class of a \pj.

Now assume that $\cc{x}_{A^{\alpha}}$ is the class of a nonzero
projection in $M_{\infty} (A^{\alpha})$. By Theorem 3.1 of
\cite{MohPhil2021TRP}, $A^{\alpha}$ is simple. Now, use Lemma 2.6 of \cite{Phl40} to choose
$b \in (M_{\infty} (A^{\alpha}))_{+} \setminus \{0\}$ such that
$\cc{b}_{A^{\alpha}} \leq \cc{a_j}_{A^{\alpha}}$ for
$j = 0, 1, \hdots, m$.
Set $k = \max ( k_0, k_1, \ldots, k_m )$.
By Lemma~\ref{L_3618_Lg36}, there exist $\mu, \eta \in \W (A^{\alpha})$,
both not the classes of \pj{s},
such that
\[
\mu \leq \cc{x}_{A^{\alpha}} \leq \mu + \eta
\andeqn
(2 k + 2) \eta \leq \cc{b}_{A^{\alpha}}.
\]
For $j = 0, 1, \hdots, m$, we then have
\[
\begin{split}
(2 k_j + 2) (\mu + \eta)
& \leq (2 k_j + 2) \cc{x}_{A^{\alpha}} + (2 k_j + 2) \eta
\\
& \leq 2 k_j \cc{a_j}_{A^{\alpha}} + \cc{b}_{A^{\alpha}}
  \leq 2 k_j \cc{a_j}_{A^{\alpha}} + \cc{a_j}_{A^{\alpha}}
  = (2 k_j + 1) \cc{a_j}_{A^{\alpha}}.
\end{split}
\]
Hence $\mu + \eta \les \cc{a_j}_{A^{\alpha}}$.
Since $\mu + \eta$ is not the class of a \pj,
the previous case yields the second step in the calculation
\[
\cc{x}_{A^{\alpha}} \leq \mu + \eta \leq \sum_{j = 0}^m \cc{b_j}_{A^{\alpha}}.
\]
This completes the proof that if $A$ has
$m$-comparison, then so does $A^{\alpha}$.
\end{proof}

Proposition~4.19 of~\cite{MohPhil2021TRP} says that if $\af$
has the restricted tracial Rokhlin property with comparison,
then so does $g \mapsto \id_{M_N} \otimes \af_g$ for any~$N$.
We need a more careful version of a consequence of that statement,
namely the matrix version of Theorem~2.14 of~\cite{MohPhil2021TRP}.
Its role in the next proof is to avoid rather awkward notation.
We will not need all parts of the conclusion.
(We give a nearly full list for  possible use elsewhere.)
In our application, $x$ and~$y$ will be in $A$ and $A^{\af}$,
embedded in $M_N (A)$ and $M_N (A^{\alpha})$ as the upper left corners.

To prove the full version, including the condition
$\| \ps (a) \| > \| a \| - \ep$ for all $a \in F_{1} \cup F_{2}$
(omitted below), the strategy is as follows.
The proof of Proposition~4.19 of~\cite{MohPhil2021TRP}
gives a map of the form $\id_{M_N} \otimes \ph_0$.
Use this map in the proof of Theorem~2.14 of~\cite{MohPhil2021TRP}.
It is awkward to write a careful proof this way,
so we give instead the version omitting the condition above,
and for which it is easy to write a proof using the conclusion of
Theorem~2.14 of~\cite{MohPhil2021TRP}.
We do not need the omitted condition.

\begin{lem}\label{L_6327_TRP_Mn}
Let $A$ be an infinite dimensional simple separable unital \ca,
let $G$ be a second countable compact group, and let
$\alpha \colon G \to \Aut (A)$ be an action which has the
restricted tracial Rokhlin property with comparison.
Let $N \in \N$.
For every $\ep > 0$, every $n \in \N$,
every compact subset $F_{1} \subseteq M_N (A)$,
every compact subset $F_2 \subseteq M_N (A^{\alpha})$,
every $x \in M_N (A)_{+} \setminus \{ 0 \}$,
and every $y \in M_N (A^{\alpha})_{+} \setminus \{ 0 \}$, there exist
a projection $p_0 \in A^{\alpha}$ and a
unital completely positive contractive map
$\psi_0 \colon A \to p_0 A^{\alpha} p_0$ such that,
taking $p = 1_{M_N} \otimes p_0 \in M_N (A^{\alpha})$
and $\ps = \id_{M_N} \otimes \psi_0 \colon M_N (A) \to p M_N (A^{\alpha}) p$,
the following hold:
\begin{enumerate}
\item\label{Item_1X07_18}
$\ps$ is an $(n, F_1 \cup F_2, \ep)$-approximately
multiplicative map in the sense of Definition~1.5 of~\cite{MohPhil2021TRP}.
\item\label{commute1469}
$\| p a - a p \| < \varepsilon$
for all $a \in F_{1} \cup F_{2}$.
\item\label{Item_1X07_21}
$\| \psi (a) - p a p \| < \ep$
for all $a \in F_2$.
\item\label{Item_1X07_19}
$1_{M_N (A)} - p \precsim_{A} x$.
\item\label{Item_1X07_20}
$1_{M_N (A)} - p \precsim_{A^{\alpha}} y$.
\item\label{Item_1X07_20p}
$1_{M_N (A)} - p \precsim_{A^{\alpha}} p$.
\end{enumerate}
\end{lem}

\begin{proof}
We will systematically write elements $b \in M_n (A)$
as matrices $b = (b_{j, k})_{j, k = 1}^N$.
We also use standard matrix unit notation, so that
$(b_{j, k})_{j, k = 1}^N = \sum_{j, k = 1}^N e_{j, k} \otimes b_{j, k}$.
Define compact subsets $S_1 \S A$ and $S_2 \S A^{\af}$ by
\[
S_{\nu} = \bigl\{ a_{j, k} \colon
 {\mbox{$a \in F_{\nu}$ and $j, k \in \{ 1, 2, \ldots, N \}$}} \bigr\}
\]
for $\nu = 1, 2$.
Following the argument in the proof of
Proposition~4.19 of~\cite{MohPhil2021TRP}, use Lemma~3.8 of~\cite{PhlT1},
then Theorem 3.1 and Proposition 3.2 of \cite{MohPhil2021TRP}
and Lemma~2.4 of~\cite{Phl40},
to get $x_0 \in A_{+} \setminus \{ 0 \}$
and $y_0 \in (A^{\alpha})_{+} \setminus \{ 0 \}$ such that
\begin{equation}\label{Eq_6410_Multiple}
1_{M_n} \otimes x_0 \precsim_{A} x
\qquad {\mbox{and}} \qquad
1_{M_n} \otimes y_0 \precsim_{A^{\alpha}} y.
\end{equation}
Further set
\[
\ep_0 = \frac{\ep}{N^{n + 1} + 1}.
\]
Apply Theorem~2.14 of~\cite{MohPhil2021TRP}
with $\ep_0$ in place of~$\ep$,
with $n$ as given,
with $S_1$ and $S_2$ in place of $F_1$ and~$F_2$,
and with $x_0$ and $y_0$ in place of $x$ and~$y$.
Let $p_0 \in A^{\alpha}$ and $\ps_0 \colon A \to p_0 A^{\alpha} p_0$
be the resulting \pj{} and $(n, S_1 \cup S_2, \ep_0)$-approximately
multiplicative unital completely positive contractive map.
Define $p = 1_{M_N} \otimes p_0$ and $\ps = \id_{M_N} \otimes \psi_0$.

We verify the conclusions of the lemma.
We start with~(\ref{Item_1X07_18}).
An argument by induction on $m$ shows that every
product $b = a_1 a_2 \cdots a_m$
with $a_1, a_2, \ldots, a_m \in F_1 \cup F_2$ has the form
$(b_{j, k})_{j, k = 1}^N$ in which, for $j, k = 1, 2, \ldots, N$,
the element $b_{j, k}$ is a sum of $N^{m - 1}$ products
of length~$m$ of elements of $S_1 \cup S_2$.
The matrix entries of $c = \ps (a_1) \ps (a_2) \cdots \ps (a_m)$
are sums of the same products, except that
each element $z \in S_1 \cup S_2$ is replaced with $\ps_0 (z)$.
When $m \in \{ 1, 2, \ldots, n \}$, the approximate multiplicativity
condition on $\ps_0$, used at the second step, implies that
\[
\begin{split}
\| (\ps (b) - c)_{j, k} \|
& = \| \ps_0 (b_{j, k}) - c_{j, k} \|
\\
& \leq N^{m - 1}
   \sup_{z_1, z_2, \ldots, z_m \in S_1 \cup S_2}
      \bigl\| \ps_0 (z_1 z_2 \cdots z_m)
            - \ps_0 (z_1) \ps_0 (z_2) \cdots \ps_0 (z_m) \bigr\|
\\
& \leq N^{m - 1} \ep_0.
\end{split}
\]
Therefore
\[
\| \ps (b) - c \| \leq N^2 \cdot N^{m - 1} \ep_0 < \ep.
\]
This proves~(\ref{Item_1X07_18}).

For~(\ref{commute1469}),
let $a = (a_{j, k})_{j, k = 1}^N \in F_1 \cup F_2$.
Then $a_{j, k} \in S_1 \cup S_2$ for $j, k = 1, 2, \ldots, N$.
So $p a - a p = \bigl( p_0 a_{j, k} - a_{j, k} p_0 \bigr)_{j, k = 1}^N$,
whence
\[
\| p a - a p \|
 \leq \sum_{j, k = 1}^N \| p_0 a_{j, k} - a_{j, k} p_0 \|
 < N^2 \ep_0 \leq \ep.
\]
Part~(\ref{Item_1X07_21}) is similar:
if $a = (a_{j, k})_{j, k = 1}^N \in F_2$, then $a_{j, k} \in S_2$
for all $j, k$, so
\[
\| \psi (a) - p a p \|
  \leq \sum_{j, k = 1}^N \| \ps_0 (a_{j, k} ) - p_0 a_{j, k} p_0 \|
   < N^2 \ep_0 \leq \ep.
\]

For~(\ref{Item_1X07_19}), we have, using~(\ref{Eq_6410_Multiple})
at the last step,
\[
\ccA{1_{M_N (A)} - p}
 = N \ccA{1 - p_0}
 \leq N \ccA{x_0}
 \leq \ccA{x},
\]
and (\ref{Item_1X07_20}) is proved similarly.
Condition~(\ref{Item_1X07_20p}) is immediate from
$1 - p_0 \precsim_{A^{\alpha}} p_0$.
\end{proof}

\begin{thm}\label{Alm_Div_TRP}
Let $A$ be a stably finite simple separable infinite dimensional
unital C*-algebra,
let $G$ be a second countable compact group,
and let $\alpha \colon G \to \Aut (A)$ be an action
which has the restricted tracial Rokhlin property with comparison
but does not have the Rokhlin property.
Let $m \in \mathbb{N}$.
Suppose $A$ is $m$-almost divisible.
Then $A^{\alpha}$ is $(2 m + 1)$-almost divisible.
\end{thm}

Unlike the other proofs of similar statements in this paper,
for this proof we seem to need to use more specific consequences
of the restricted tracial Rokhlin property with comparison.

\begin{proof}[Proof of Theorem~\ref{Alm_Div_TRP}]
We verify the criterion of Proposition~\ref{P_6324_WDiv}.
Thus, let $n \in \N$, let $a \in M_{n} (A^{\af})_+$,
let $k \in \Nz$, and let $\ep > 0$.
Using the notation of Proposition~\ref{P_6324_WDiv}, we must find
$b \in M_{\I} (A^{\alpha})_+$
such that in $\W (A^{\alpha})$ we have
\begin{equation}\label{Eq_6325_Goal}
k \cc{b}_{A^{\alpha}} \leq \cc{a}_{A^{\alpha}}
\andeqn
\cc{(a - \ep)_+}_{A^{\alpha}} \leq (k + 1) (2 m + 2) \cc{b}_{A^{\alpha}}.
\end{equation}
By Proposition~4.19 of~\cite{MohPhil2021TRP},
the action $g \mapsto \id_{M_n} \otimes \af_g$ of $G$ on $M_{n} \otimes A$
also has the restricted tracial Rokhlin property with comparison.
Therefore we may assume $n = 1$.
We may also clearly assume that $(a - \ep)_+ \neq 0$.

First assume that $0$ is a limit point of $\spec (a)$.
In this case we will prove (\ref{Eq_6325_Goal}) with $m + 1$
in place of $2 m + 2$ in the second inequality.
Choose  $\rh \in (0, \frac{\ep}{3})$ such that
$\spec (x) \cap \bigl( \rh, \frac{\ep}{3} \bigr) \neq \E$.
Choose $\ld \in \spec (x) \cap \bigl( \rh, \frac{\ep}{3} \bigr)$.
Choose a \cfn{} $h \colon [0, \I) \to [0, 1]$ such that
$\supp (h) \S \bigl( \rh, \frac{\ep}{3} \bigr)$ and $h (\ld) \neq 0$. By Theorem 3.1 and Proposition 3.2 of \cite{MohPhil2021TRP}, $A^{\alpha}$ is simple and not type I. Now, Use Lemma~2.4 of~\cite{Phl40} to find $t \in ({\ov{h (a) A^{\af} h (a)}})_{+} \SM \{ 0 \}$ such that
\begin{equation}\label{Eq_6325_t}
k \cc{t}_{A^{\af}} \leq \cc{h (a)}_{A^{\af}}.
\end{equation}

Apply the other direction of Proposition~\ref{P_6324_WDiv} to~$A$.
We get $r \in \N$ and $c \in M_r (A)_{+}$ such that
in $\W (A)$ we have
\begin{equation}\label{Eq_6324_c}
k \cc{c}_{A} \leq \inn{ \left( a - \frac{\ep}{3} \right)_+}_{A}
\quad {\mbox{and}} \quad
\inn{\left( a - \frac{2 \ep}{3} \right)_+}_{A}
 \leq (k + 1) (m + 1) \cc{c}_{A}.
\end{equation}
Using Proposition~4.19 of~\cite{MohPhil2021TRP} again,
we may replace $A$ with $M_r (A)$,
and thus assume that both $a$ and~$c$ are in~$A$.

Define $N = (m + 1) (k + 1)$ and,
in the first of these using the inclusion in the upper left corner,
\[
x = 1_{k} \otimes c \in M_{k} (A) \S M_N (A)
\andeqn
y = 1_{N} \otimes c \in M_N (A).
\]
Abusing notation, we also take $A \S M_N (A)$
via the inclusion in the upper left corner.
Thus, for example, in the following, using standard matrix units in $M_N$,
$a$ is really $e_{1, 1} \otimes a$
and $c$ is really $e_{1, 1} \otimes c$.
Use~(\ref{Eq_6324_c}) to choose $w \in M_N (A)$ such that
\begin{equation}\label{Eq_6325_wsyw}
\left\| w^* y w - \left( a - \frac{2 \ep}{3} \right)_+ \right\|
  < \frac{\ep}{12}.
\end{equation}
Define
\begin{equation}\label{Eq_6325_dt}
\dt = \frac{\ep}{24 \| w \|^2 + 1}.
\end{equation}
Use the first part of~(\ref{Eq_6324_c}) to choose $v \in M_N (A)$ such that
\begin{equation}\label{Eq_6325_vsav}
\left\| v^* \left( a - \frac{\ep}{3} \right)_+ v - x \right\| < \dt.
\end{equation}
Define
\begin{equation}\label{Eq_6325_rh}
\rh = \min \left( \frac{\ep}{12}, \,  \frac{\dt}{\| v \|^2 + 1} \right).
\end{equation}

Define finite subsets $F_1 \S M_N (A)$ and $F_2 \S M_N (A^{\af})$ by
\[
F_1 = \bigl\{ c, v, w, x, y \bigr\}
\andeqn
F_2 = \left\{ \left( a - \frac{\ep}{3} \right)_+, \,
          \left( a - \frac{2 \ep}{3} \right)_+, \, (a - \ep)_{+} \right\}.
\]
Apply Lemma~\ref{L_6327_TRP_Mn} with $\rh$ in place of~$\ep$,
with $n = 3$, with $F_1$ and $F_2$ as given,
with $1_{M_N (A)}$ in place of~$x$,
and with $t$ (as in~(\ref{Eq_6325_t}), identified with $e_{1, 1} \otimes t$)
in place of~$y$,
getting a projection $p_0 \in A^{\alpha}$ and a
unital completely positive contractive map
$\psi_0 \colon A \to p_0 A^{\alpha} p_0$ such that,
taking $p = 1_{M_N} \otimes p_0 \in M_N (A^{\alpha})$
and $\ps = \id_{M_N} \otimes \ps \colon M_N (A) \to p M_N (A^{\alpha}) p$,
the following hold:
\begin{enumerate}
\item\label{I_6324_Item_1X07_18}
Whenever $m \in \{ 1, 2, 3 \}$
and $z_1, z_2, \ldots, z_m \in F_1 \cup F_2$, we have
\[
\bigl\| \psi (z_1 z_2 \cdots z_m)
  - \psi (z_1) \psi (z_2) \cdots \psi (z_m) \bigr\|
   < \rh.
\]
\item\label{I_6324_commute1469}
$\| p z - z p \| < \rh$ for all $z \in F_{1} \cup F_{2}$.
\item\label{I_6324_Item_1X07_21}
$\| \psi (z) - p z p \| < \rh$ for all $z \in F_2$.
\item\label{I_6324_Item_1X07_20}
$1 - p \precsim_{A^{\alpha}} t$.
\end{enumerate}
Given our identification of $A$ as the upper left corner of $M_N (A)$,
we are identifying $p_0 a p_0$ with $p a p$, $\ps (a)$ with $\ps_0 (a)$,
etc.

Define $b_0 = (\ps (c) - 2 \dt)_{+} \in A^{\alpha}$.
This gives, in $M_N (A^{\alpha})$,
\begin{equation}\label{Eq_6325_b0_mult}
1_k \otimes b_0 = (\ps (x) - 2 \dt)_{+}
\andeqn
1_{(m + 1) (k + 1)} \otimes b_0 = (\ps (y) - 2 \dt)_{+}.
\end{equation}
Further define
\begin{equation}\label{Eq_6325_b_Dfn}
b = b_0 \oplus t.
\end{equation}

We first claim that
\[
\left( p \left( a - \frac{2 \ep}{3} \right)_+ p - \frac{\ep}{3} \right)_+
\precsim_{A^{\af}} (\ps (y) - 2 \dt)_{+}.
\]
To prove the claim, first estimate,
using (\ref{I_6324_Item_1X07_18}), (\ref{Eq_6325_wsyw}),
and (\ref{I_6324_Item_1X07_21}) at the second step,
and (\ref{Eq_6325_rh}) at the third step,
\[
\begin{split}
& \left\| \ps (w)^* \ps (y) \ps (w)
  - p \left( a - \frac{2 \ep}{3} \right)_+ p \right\|
\\
& \hspace*{3em} {\mbox{}}
 \leq \| \ps (w)^* \ps (y) \ps (w) - \ps (w^* y w) \|
    + \left\| \ps \left( w^* y w
         - \left( a - \frac{2 \ep}{3} \right)_+ \right) \right\|
 \\
& \hspace*{6em} {\mbox{}}
  + \left\| \ps \left( \left( a - \frac{2 \ep}{3} \right)_+ \right)
     - p \left( a - \frac{2 \ep}{3} \right)_+ p \right\|
\\
& < \rh + \frac{\ep}{12} + \rh
  \leq \frac{\ep}{4}.
\end{split}
\]
Since $\| (\ps (y) - 2 \dt)_{+} - \ps (y) \| \leq 2 \dt$,
we then get, using (\ref{Eq_6325_dt}) at the third step,
\[
\begin{split}
& \left\| \ps (w)^* (\ps (y) - 2 \dt)_{+} \ps (w)
  - p \left( a - \frac{2 \ep}{3} \right)_+ p \right\|
\\
& \hspace*{3em} {\mbox{}}
 \leq 2 \dt \| w \|^2
 + \left\| \ps (w)^* \ps (y) \ps (w)
  - p \left( a - \frac{2 \ep}{3} \right)_+ p \right\|
 < \frac{\ep}{12} + \frac{\ep}{4}
  = \frac{\ep}{3}.
\end{split}
\]
Therefore
\[
\left( p \left( a - \frac{2 \ep}{3} \right)_+ p - \frac{\ep}{3} \right)_+
\precsim_{A^{\af}} \ps (w)^* (\ps (y) - 2 \dt)_{+} \ps (w)
\precsim_{A^{\af}} (\ps (y) - 2 \dt)_{+},
\]
proving the claim.

Our second claim is that
\begin{equation}\label{Eq_6325_bsub}
\cc{(a - \ep)_{+}}_{A^{\alpha}} \leq N \cc{b}_{A^{\alpha}}.
\end{equation}
To prove this, use Lemma~1.8 of~\cite{Phl40} at the second step,
the previous claim and (\ref{I_6324_Item_1X07_20}) at the third step,
and (\ref{Eq_6325_b0_mult}) at the fourth step,
getting
\[
\begin{split}
(a - \ep)_{+}
& = \left( \left( a - \frac{2 \ep}{3} \right)_+ - \frac{\ep}{3} \right)_+
  \precsim_{A^{\af}}
 \left( p \left( a - \frac{2 \ep}{3} \right)_+ p - \frac{\ep}{3} \right)_+
 \oplus (1 - p)
\\
& \precsim_{A^{\af}} (\ps (y) - 2 \dt)_{+} \oplus t
  = 1_{N} \otimes b_0 \oplus t
  \precsim_{A^{\af}} 1_{N} \otimes (b_0 \oplus t)
  = 1_{N} \otimes b.
\end{split}
\]
The claim follows.

This claim is the second part of~(\ref{Eq_6325_Goal}),
with $m + 1$ in place of $2 m + 2$.

We next claim that
\begin{equation}\label{Eq_6325_xsubb}
(\ps (x) - 2 \dt)_{+} \precsim_{A^{\af}} \left( a - \frac{\ep}{3} \right)_+.
\end{equation}
Using $\ps (v) \in p M_N (A^{\af}) p$ at the first step,
using (\ref{I_6324_Item_1X07_21}), (\ref{I_6324_Item_1X07_18}),
and (\ref{Eq_6325_vsav}) at the third step,
and using (\ref{Eq_6325_rh}) at the fourth step,
we get
\[
\begin{split}
& \left\| \ps (v)^* \left( a - \frac{\ep}{3} \right)_+ \ps (v)
    - \ps (x) \right\|
\\
& \hspace*{3em} {\mbox{}}
 = \left\| \ps (v)^* p \left( a - \frac{\ep}{3} \right)_+ p \ps (v)
    - \ps (x) \right\|
\\
& \hspace*{3em} {\mbox{}} \leq \| \ps (v) \|^2
  \left\| p \left( a - \frac{\ep}{3} \right)_+ p
    - \ps \left( \left( a - \frac{\ep}{3} \right)_+ \right) \right\|
\\
& \hspace*{6em} {\mbox{}}
  + \left\|
    \ps (v)^* \ps \left( \left( a - \frac{\ep}{3} \right)_+ \right) \ps (v)
    - \ps \left( v^* \left( a - \frac{\ep}{3} \right)_+ v \right) \right\|
\\
& \hspace*{6em} {\mbox{}}
  + \left\| v^* \left( a - \frac{\ep}{3} \right)_+ v - x \right\|
\\
& \leq \| v \|^2 \rh + \rh + \dt \leq 2 \dt.
\end{split}
\]
Therefore
\[
( \ps (x) - 2 \dt)_{+}
 \precsim_{A^{\af}} \ps (v)^* \left( a - \frac{\ep}{3} \right)_+ \ps (v)
 \precsim_{A^{\af}} \left( a - \frac{\ep}{3} \right)_+,
\]
as claimed.

We now use~(\ref{Eq_6325_b_Dfn})~at the first step,
(\ref{Eq_6325_b0_mult}) at the second step,
the previous claim and~(\ref{Eq_6325_t}) at the third step,
and $h (a) \perp \left( a - \frac{\ep}{3} \right)_+$ at the fourth step,
to get
\[
\begin{split}
1_{k} \otimes b
& = 1_{k} \otimes (b_0 \oplus t)
\\
& = (\ps (x) - 2 \dt)_{+} \oplus 1_{k} \otimes t
  \precsim_{A^{\af}} \left( a - \frac{\ep}{3} \right)_+
    \oplus h (a)
  \precsim_{A^{\af}} a.
\end{split}
\]
This is the first inequality in~(\ref{Eq_6325_Goal}),
and completes the proof when $0$ is a limit point of $\spec (a)$.

Now assume $0$ is not a limit point of $\spec (a)$.
We may assume that $a$ is a \pj.
Lemma~\ref{L_5Z26_InterpPj} yields
$x, y \in (a (A^{\af} \otimes K) a)_{+} \S \Mi (A^{\af})$
and $\ep \in (0, 1)$
such that $\spec (x) = \spec (y) = [0, 1]$,
$a \precsim_{A^{\af}} (x - \ep)_{+} \oplus (y - \ep)_{+}$,
and $y \precsim_{A^{\af}} (x - \ep)_{+}$.
Applying the case already done to~$x$,
we get $b \in M_{\I} (A^{\alpha})_+$ such that
\[
k \cc{b}_{A^{\alpha}} \leq \cc{x}_{A^{\alpha}}
\andeqn
\cc{(x - \ep)_+}_{A^{\alpha}} \leq (k + 1) (m + 1) \cc{b}_{A^{\alpha}}.
\]
Now
\[
k \cc{b}_{A^{\alpha}} \leq \cc{x}_{A^{\alpha}} \leq \cc{a}_{A^{\alpha}}
\]
and, at the first step using the fact that $a$ is a \pj,
\[
\begin{split}
\cc{(a - \ep)_{+}}_{A^{\alpha}}
& = \cc{a}_{A^{\alpha}}
  \leq \cc{(x - \ep)_{+}}_{A^{\alpha}} + \cc{y}_{A^{\alpha}}
\\
& \leq 2 \cc{(x - \ep)_{+}}_{A^{\alpha}}
  \leq 2 (k + 1) (m + 1) \cc{b}_{A^{\alpha}}.
\end{split}
\]
This completes the proof that $A^{\alpha}$ is $(2 m + 1)$-almost divisible.
\end{proof}

Now, we prove the main theorem of this section.

\begin{thm}\label{T_6326_TRP_pure}
Let $A$ be a simple unital separable infinite dimensional stably finite
C*-algebra.
Let $\alpha \colon G \to \Aut(A)$ be an action of a second countable
compact group which has the tracial Rokhlin property with comparison
but does not have the Rokhlin property.
If $A$ is pure, then the crossed product
$C^* \bigl( G, \, A, \, \af \bigr)$ is pure.
\end{thm}

If $G$ is a compact Lie group, and $\af$ has the Rokhlin property,
then $C^* \bigl( G, \, A, \, \af \bigr)$ is pure by
Theorem~\ref{T_6325_RDim}.

\begin{proof}[Proof of Theorem~\ref{T_6326_TRP_pure}]
It follows from Proposition~\ref{Alm_Div_TRP},
Proposition~\ref{ComparisonTRP},
and Theorem~\ref{T_PureEquiv} that $A^{\alpha}$ is pure.
Now, Corollary~3.9 of~\cite{MohPhil2021TRP} implies that
$C^* \bigl( G, \, A, \, \af \bigr)$ is strongly
Morita equivalent to $A^{\alpha}$. Therefore
$C^* \bigl( G, \, A, \, \af \bigr)$ is pure.
\end{proof}

\section{Examples}\label{Sec_Examples}

We give examples of crossed product C*-algebras of the
form $C^* \bigl( \Z, \, C (X, D), \, \af \bigr)$, in which $D$
is a simple unital pure C*-algebra, $X$ is a compact metric
space, and $\af$ induces a minimal homeomorphism
$h \colon X \to X$.
Using Theorem~\ref{T_4_PureCrossProd},
we will conclude they are pure, and hence have stable rank
one by Corollary~\ref{C_5_CrossProdSR1}.
In some cases, we
will additionally be able to conclude the crossed product
has real rank zero.
Several of these examples are
adapted from examples in Section 6 of \cite{ArBcPh2}.
The methods used there do not apply to the examples here,
either because $D$ is not $\JS$-stable,
the space $X$ does not have dimension 0 or 1, or both.

Many of the examples are based on reduced C*-algebras of free
groups, so we establish notation for these separately.

\begin{ntn}\label{N_1_QFreeFn}
For $n \in \N \cup \{ \I \}$, we let $F_n$ denote the free
group on $n$~generators.
We let $u_1, u_2, \ldots, u_n \in C^*_{\mathrm{r}} (F_n)$
(when $n = \I$,
$u_1, u_2, \ldots \in C^*_{\mathrm{r}} (F_{\infty})$)
be the ``standard'' unitaries in $C^*_{\mathrm{r}} (F_n)$,
obtained as the images of the standard generators of~$F_n$.
For $\zt = (\zt_1, \zt_2, \ldots, \zt_n) \in (S^1)^n$
(when $n = \I$,
for $\zt = (\zt_n)_{n \in \N} \in (S^1)^{\N}$), we let
$\ph_{\zt} \in \Aut ( C^*_{\mathrm{r}} (F_n) )$ be the
(quasifree) automorphism determined by
$\ph_{\zt} (u_k) = \zt_k u_k$ for $k = 1, 2, \ldots, n$
(when $n = \I$, for $k = 1, 2, \ldots)$.
It is well known
that $\zt \mapsto \ph_{\zt}$ is \ct.
\end{ntn}

We can also treat $C^*_{\mathrm{r}} (F_{\infty})$ differently.

\begin{ntn}\label{N_2_FreeShift}
Take the standard generators of the free group $F_{\infty}$
to be indexed by~$\Z$, and for $n \in \Z$ let
$u_n \in C^*_{\mathrm{r}} (F_{\infty})$ be the unitary
obtained as the image of the corresponding generator
of~$F_{\infty}$.
We denote by $\sm$ the free shift on
$C^*_{\mathrm{r}} (F_{\infty})$, that is, the automorphism
$\sm \in \Aut ( C^*_{\mathrm{r}} (F_{\infty}) )$ determined
by $\sm (u_n) = u_{n + 1}$ for $n \in \Z$.
\end{ntn}

\begin{prp}\label{P_FreePure}
For $n \in \N \cup \{ \I \}$, let $F_n$ be the free
group on $n$~generators. Then $C^*_{\mathrm{r}} (F_n)$ is pure.
\end{prp}

\begin{proof}
Theorem~A of~\cite{AGKEP} implies that
$C^*_{\mathrm{r}} (F_n)$ has strict comparison.
Pureness follows by Theorem~5.13 of~\cite{AnPrTlVt}, as
described in Example~5.14 there.
\end{proof}

The first example is based on Example~6.3 of \cite{ArBcPh2},
replacing the $2^{\infty}$ UHF algebra
$D = \bigotimes_{n \in \N} M_{2}$ used there with the pure,
non $\JS$-stable algebra $D = C^*_{\mathrm{r}} (F_{\infty})$.

\begin{exa}\label{E_1_ShiftAndFInf}
Set $X_0 = (S^1)^{\Z}$.
Let $h_0 \colon X_0 \to X_0$ be the
forward shift, defined for $\zt = (\zt_n)_{n \in \Z} \in X_0$
by $h_0 (\zt) = (\zt_{n - 1})_{n \in \Z}$.
Let $D = C^*_{\mathrm{r}} (F_{\infty})$.
Choose a bijection $\sm \colon \N \to \Z$, and for
$\zt = (\zt_n)_{n \in \Z} \in X_0$ let
$\af_{\zt}^{(0)} \in \Aut (D)$ be the automorphism $\ph_{\zt}$ of
Notation~\ref{N_1_QFreeFn}.
Then $\zt \mapsto \af_{\zt}^{(0)}$ is \ct.

We have $\mdim (h_0) = 1$ by Proposition~3.3 of~\cite{LndWs},
and we can use $X_0$, $h_0$, $D$, and
$\zt \mapsto \af^{(0)}_{\zt}$ in Lemma~\ref{L_1_GetAuto}.
However, $h_0$ is not minimal.
We therefore proceed as in Example~6.3 of \cite{ArBcPh2}.
Identify $[0, 1]$ with a closed arc in~$S^1$,
say via $\ld \mapsto \exp (\pi i \ld)$.
Use this to identify
$[0, 1]^{\Z}$ with a closed subset of $X_0 = (S^1)^{\Z}$.
This identification is equivariant when both spaces are
equipped with the shift \hme{s}.
Let $X \S X_0$ and
$h = h_0 |_{X} \colon X \to X$ be the minimal subshift
in~\cite{GlKr}, which can be taken to have mean dimension
arbitrarily close to~$1$ (by Proposition 3.3
of~\cite{LndWs} and the argument in the proof
of Proposition 3.5 of~\cite{LndWs}).
For $\zt \in X$, let
$\af_{\zt} = \af_{\zt}^{(0)}$.
Let
$\af \colon \Z \to \Aut (C (X, D))$ be the corresponding
action as in Lemma~\ref{L_1_GetAuto}.
Then $\af$ lies
over the free minimal action of $\Z$ on~$X$ generated by~$h$.

The crossed product
$C^* \bigl( \Z, \, C (X, D), \, \af \bigr)$ is simple
by Proposition~1.6 of \cite{ArBcPh2} and $D$ is pure by
Proposition~\ref{P_FreePure},
so the crossed product is pure by Theorem~\ref{T_4_PureCrossProd}
and has stable rank one by
Corollary~\ref{C_5_CrossProdSR1}.
\end{exa}

The next example generalizes Example~6.3 of~\cite{ArBcPh2}
by taking the reduced free product of $D$ there with another
C*-algebra.

\begin{exa}\label{E_2_ShiftAndFInf}
Let $X_0$, $h_0$, $X$, and $h$
be as in Example~\ref{E_1_ShiftAndFInf}.
Let $D_0 = \bigotimes_{n \in \N} M_2$ be the $2^{\infty}$~UHF algebra.
Choose a bijection $\sm \colon \N \to \Z$,
and for $\zt = (\zt_n)_{n \in \Z} \in X_0$ define
\[
\af_{\zt}^{(0)}
 = \bigotimes_{n \in \N}
  \Ad \left( \left( \begin{matrix}
       1 & 0 \\
       0 & \zt_{\sm (n)}
     \end{matrix} \right) \right)
 \in \Aut (D_0).
\]
Then $\zt \mapsto \af_{\zt}^{(0)}$ is \ct.
Set $D = D_0 \ast_{r} C ([0, 1])$,
taking the reduced free product to be amalgamated over $\C \cdot 1$
and with respect to the unique \tst{} on~$D_0$
and the Lebesgue measure state on $C ([0, 1])$.
For $\zt \in X$,
define $\af_{\zt} = \af_{\zt}^{(0)} \ast_{r} \id_{C ([0, 1])}$.
The crossed product $C^* \bigl( \Z, \, C (X, D), \, \af \bigr)$ is
simple by Proposition~1.6 of \cite{ArBcPh2}.
Examples 3 and 8 of~\cite{HKER}, together with Theorem~B there,
imply that $D$ is selfless in the sense of Definition~2.1
of~\cite{Rob}. Theorem~3.1 of~\cite{Rob} implies that $D$ has
stable rank one, strict comparison and a unique normalized
2-quasitrace. It follows that $D$ is pure by Theorem~5.13
of~\cite{AnPrTlVt}, and so the crossed product is pure by
Theorem~\ref{T_4_PureCrossProd} and has stable rank one
by Corollary~\ref{C_5_CrossProdSR1}.
\end{exa}

In Example~\ref{E_2_ShiftAndFInf}, the algebra $C ([0, 1])$
can be replaced with an infinite free product, or also by an algebra
which is not even exact, provided that the conditions of Theorem~B
of~\cite{HKER} are satisfied. See the examples in Section 3.3 there
for various possibilities.

The next example is similar to Example~\ref{E_1_ShiftAndFInf},
but with larger mean dimension.
It generalizes Example~6.4 of~\cite{ArBcPh2}.

\begin{exa}\label{E_3_ShiftAndFInf}
Let $X_0$, $h_0$, $D$, and $\zt \mapsto \af_{\zt}^{(0)}$
be as in Example~\ref{E_1_ShiftAndFInf}.
We will, however, use the \hme{} $h_0^2$ of~$X_0$,
which is the shift on $(S^1 \times S^1)^{\Z}$.
Let $(X, h)$ be the minimal subspace of the shift on
$([0, 1]^2)^{\Z}$
constructed in Proposition 3.5 of \cite{LndWs},
which satisfies $\mdim (h) > 1$.
Use an embedding of $[0, 1]^2$ in~$(S^1)^2$
to choose an equivariant \hme{} $g$
from $(X, h)$ to an invariant closed subset of $(X_0, h_0^2)$.
For $x \in X$, let $\af_{x} = \af_{g (x)}^{(0)}$.
Let $\af \colon \Z \to \Aut (C (X, D))$
be the corresponding action as in Lemma~\ref{L_1_GetAuto}.
The crossed product
$C^* \bigl( \Z, \, C (X, D), \, \af \bigr)$ is
simple by Proposition~1.6 of \cite{ArBcPh2} and $D$ is
pure, so the crossed product is pure by Theorem~\ref{T_4_PureCrossProd}
and has stable rank one by
Corollary~\ref{C_5_CrossProdSR1}.
\end{exa}

The next example generalizes Example~6.7 of \cite{ArBcPh2}.
There, it was necessary to assume that $X$ is the Cantor
set.

\begin{exa}\label{E_4_FreeShift}
Let $X$ be any compact metric space, and let $h$ be an
arbitrary minimal homeomorphism of~$X$.
Let
$\sm \in \Aut ( C^*_{\mathrm{r}} (F_{\infty}) )$ be as
in Notation~\ref{N_2_FreeShift}.
Let
$\af \in \Aut \bigl( C (X) \otimes C^*_{\mathrm{r}} (F_{\infty}) \bigr)$
be the tensor product of the automorphism
$f \mapsto f \circ h^{-1}$ of $C (X)$ and~$\sm$.
Then
$C^* \bigl( \Z, \, C (X) \otimes C^*_{\mathrm{r}} (F_{\infty}),
  \, \af \bigr)$
is simple by Proposition~1.6 of \cite{ArBcPh2}
and $D$ is pure by Proposition~\ref{P_FreePure},
so the crossed product is pure by Theorem~\ref{T_4_PureCrossProd}
and has stable rank one by Corollary~\ref{C_5_CrossProdSR1}.

If $C^* (\Z, X, h)$ has real rank zero,
for example if $h$ is an irrational rotation, then
$C^* \bigl( \Z, \, C (X) \otimes C^*_{\mathrm{r}} (F_{\infty}),
  \, \af \bigr)$
has real rank zero by Corollary~\ref{C_6327_RRZ}.
\end{exa}

The next several examples give actions which can be viewed
as various forms of noncommutative Furstenberg transformations.
They are more natural versions of this idea than the
actions in Examples 6.8, 6.9 and 6.12 of \cite{ArBcPh2}.

\begin{exa}\label{E_5_IrrRot}
Fix any $n \in \set{ 2, 3, \ldots, \infty }$ and take
$D = C^*_{\mathrm{r}} (F_{n})$.
Adopt Notation~\ref{N_1_QFreeFn}.
Take $X = S^{1}$, let
$\te \in \R \setminus \Q$ and let $h \colon X \to X$ be the
irrational rotation  $\zt \mapsto e^{2 \pi i \te} \zt$.
Following Notation~\ref{N_1_QFreeFn} for the
generators of~$D$, for $\zt \in X$ let
$\af_{\zt} \in \Aut (D)$ be determined by
$\af_{\zt} (u_k) = \zt u_k$ for
$k = 1, 2, \ldots, n$ (or $k \in \N$ if $n = \I$).
Apply Lemma~\ref{L_1_GetAuto} to get an action
$\af \colon \Z \to \Aut \big( C (X, D ) \big)$.
The algebra
$C^* \bigl( \Z, \, C (X, D ), \, \af \bigr)$
is simple by Proposition~1.6 of \cite{ArBcPh2} and
$D$ is pure by Proposition~\ref{P_FreePure},
so the crossed product is pure by
Theorem~\ref{T_4_PureCrossProd} and has stable
rank one by Corollary~\ref{C_5_CrossProdSR1}.
It is well known that the irrational rotation
algebra $C^* ( \Z, X, h )$ has real rank zero.
(The range of the canonical map is the set
$\set{m + n \theta \colon m, n, \in \Z}$, which is
dense in $\R$ since $\theta$ is irrational.)
It follows from Corollary~\ref{C_6327_RRZ} that
$C^* \bigl( \Z, \, C (X, D ), \, \af \bigr)$ has real rank zero.
\end{exa}

\begin{exa}\label{E_6_MultiFrst}
Fix $n \in \{ 2, 3, \ldots, \I \}$ and take
$D = C^*_{\mathrm{r}} (F_{n})$.
Fix $\te_1, \te_2, \ldots, \te_n \in \R$
(or $\te_1, \te_2, \ldots \in \R$)
such that $1, \te_1, \te_2, \ldots, \te_n$
(or $1, \te_1, \te_2, \ldots$)
are linearly independent over~$\Q$.
Take $X = (S^{1})^{n}$, and let $h \colon X \to X$ be
\[
h (\ld_1, \ld_2, \ldots, \ld_n)
 = \bigl( e^{2 \pi i \te_1} \ld_1, \, e^{2 \pi i \te_2} \ld_2,
 \, \ldots, \, e^{2 \pi i \te_n} \ld_n \bigr)
\]
for $\ld_1, \ld_2, \ldots, \ld_n \in S^1$,
or, if $n = \I$,
\[
h (\ld_1, \ld_2, \ldots)
 = \bigl( e^{2 \pi i \te_1} \ld_1, \, e^{2 \pi i \te_2} \ld_2,
 \, \ldots, \bigr)
\]
for $\ld_1, \ld_2, \ldots \in S^1$.
Following
Notation~\ref{N_1_QFreeFn} for the generators of~$D$,
for $\zt \in X$ let $\af_{\zt} \in \Aut (D)$
(called $\ph_{\zt}$ in Notation~\ref{N_1_QFreeFn}) be
determined by $\af_{\zt} (u_k) = \zt_k u_k$ for
$k = 1, 2, \ldots, n$ (or $k \in \N$ if $n = \I$).
Apply Lemma~\ref{L_1_GetAuto} to get an action
$\af \colon
\Z \to \Aut \bigl( C (X, \, C^*_{\mathrm{r}} (F_{n}) ) \bigr)$.
Then
$C^* \bigl( \Z, \, C (X, \, C^*_{\mathrm{r}} (F_{n}) ), \, \af \bigr)$
is simple, pure, and has stable rank one as in
Example~\ref{E_5_IrrRot}.
The transformation group C*-algebra
$C^* ( \Z, X, h )$ again has real rank zero.
The range of the canonical map
$\rho \colon K_{0} (C^* ( \Z, X, h )) \to \Aff (T (C^* ( \Z, X, h )))$
is the set
\[
\set{m_0 + m_1 \theta_1 + \cdots + m_n \theta_n
\colon m_0, m_1, \ldots, m_n \in \Z},
\]
(analogously if $n = \I$)
which is dense in $\R$ so long as $\theta_k$ is irrational
for at least one~$k$.
(Linear independence over $\Q$ of
$1, \te_1, \te_2, \ldots, \te_n$, or of $1, \te_1, \te_2, \ldots$,
is needed for minimality, but not here.)
It follows from Corollary~\ref{C_6327_RRZ} that
$C^* \bigl( \Z, \, C (X, D ), \, \af \bigr)$ has real rank zero.
\end{exa}

\begin{exa}\label{E_7_CSrFGp_Furst}
Let $X$ be any compact metric space, and let $h$ be an
arbitrary minimal homeomorphism of~$X$.
Following Notation~\ref{N_1_QFreeFn} for the generators
of $C^*_{\mathrm{r}} (F_{2})$, let
$\gm \in \Aut ( C^*_{\mathrm{r}} (F_{2}) )$ be
given by $\gm (u_1) = u_1$ and $\gm (u_2) = u_1 u_2$.
Let
$\af \in \Aut \big( C (X) \otimes C^*_{\mathrm{r}} (F_{2}) \big)$
be the tensor product of the automorphism
$f \mapsto f \circ h^{-1}$ of $C (X)$ and~$\gm$.
The algebra
$C^* \big( \Z, \, C (X, D ), \, \af \big)$
is simple by Proposition~1.6 of \cite{ArBcPh2} and
$D$ is pure by Proposition~\ref{P_FreePure},
so $C^* \big( \Z, \, C (X, D ), \, \af \big)$ is pure by
Theorem~\ref{T_4_PureCrossProd} and has stable rank
one by Corollary~\ref{C_5_CrossProdSR1}.
\end{exa}

The mapping $u_1 \mapsto u_1$ and $u_2 \mapsto u_1 u_2$
really is an automorphism of $F_2$, because it has the
inverse $u_1 \mapsto u_1$ and $u_2 \mapsto u_1^{-1} u_2$.

\end{document}